\documentclass{amsart}

\usepackage{amssymb}
\usepackage{times}
\usepackage{amsmath}
\usepackage{graphicx}

\newcommand{\erase}[1]{}

\newtheorem{theorem}{Theorem}[section]
\newtheorem{lemma}[theorem]{Lemma}
\newtheorem{proposition}[theorem]{Proposition}
\newtheorem{corollary}[theorem]{Corollary}

\newtheorem{_algorithm}[theorem]{Algorithm}

\newtheorem{_procedure}[theorem]{Procedure}

\newtheorem{_definition}[theorem]{Definition}
\newenvironment{definition}{\begin{_definition}\rm}{\end{_definition}}

\newtheorem{_remark}[theorem]{\it Remark}
\newenvironment{remark}{\begin{_remark}\rm}{\end{_remark}}

\newtheorem{_example}[theorem]{Example}
\newenvironment{example}{\begin{_example}\rm}{\end{_example}}

\newtheorem{_assumption}[theorem]{Assumption}

\newtheorem{_construction}[theorem]{Construction}

\newtheorem{_claim}[theorem]{Claim}

\newtheorem{_conjecture}[theorem]{Conjecture}

\numberwithin{equation}{section}
\numberwithin{table}{section}
\numberwithin{figure}{section}

\newcommand{\C}{\mathord{\mathbb C}}
\newcommand{\F}{\mathord{\mathbb F}}

\renewcommand{\P}{\mathord{\mathbb  P}}
\newcommand{\Q}{\mathord{\mathbb  Q}}
\newcommand{\R}{\mathord{\mathbb R}}

\newcommand{\Z}{\mathord{\mathbb Z}}

\newcommand{\CCC}{\mathord{\mathcal C}}
\newcommand{\DDD}{\mathord{\mathcal D}}

\newcommand{\III}{\mathord{\mathcal I}}
\newcommand{\KKK}{\mathord{\mathcal K}}
\newcommand{\LLL}{\mathord{\mathcal L}}

\newcommand{\OOO}{\mathord{\mathcal O}}
\newcommand{\PPP}{\mathord{\mathcal P}}

\newcommand{\RRR}{\mathord{\mathcal R}}
\newcommand{\SSS}{\mathord{\mathcal S}}
\newcommand{\TTT}{\mathord{\mathcal T}}

\newcommand{\VVV}{\mathord{\mathcal V}}
\newcommand{\WWW}{\mathord{\mathcal W}}

\newcommand{\maprightsp}[1]{\; \smash{\mathop{\; \longrightarrow \; }\limits\sp{#1}}\; }
\newcommand{\maprightsb}[1]{\; \smash{\mathop{\; \longrightarrow \; }\limits\sb{#1}}\; }

\newcommand{\mapdown}{\phantom{\Big\downarrow}\hskip -8pt \downarrow}
\newcommand{\mapdownright}[1]{\mapdown\rlap{$\vcenter{\hbox{$\scriptstyle#1$}}$}}
\newcommand{\mapdownleft}[1]{\rlap{$\vcenter{\hbox{$\scriptstyle#1$}}$}%
\phantom{\Big\downarrow}\mapdown}
\newcommand{\mapdownsurj}{
\hbox{$\bigm\downarrow$}
\llap{\hbox{\raise 2pt\hbox{$\bigm\downarrow$}}}%
\vstrechmapdown
}

\newcommand{\inj}{\hookrightarrow}

\newcommand{\isom}{\mathbin{\,\raise -.6pt\rlap{$\to$}\raise 3.5pt \hbox{\hskip .3pt$\mathord{\sim}$}\,}}

\newcommand{\set}[2]{\{\; {#1} \; \mid \; {#2} \;  \}}
\newcommand{\shortset}[2]{\{ {#1} \,|\, {#2}   \}}

\newcommand{\gen}[1]{\langle {#1}  \rangle}

\newcommand{\tensor}{\otimes}

\newcommand{\sprime}{\sp\prime}

\newcommand{\spprime}{\sp{\prime\prime}}

\newcommand{\sperp}{\sp{\perp}}

\newcommand{\dual}{\sp{\vee}}

\newcommand{\inv}{\sp{-1}}

\newcommand{\Hom}{\mathord{\mathrm {Hom}}}

\newcommand{\GL}{\mathord{\mathrm {GL}}}

\newcommand{\PGU}{\mathord{\mathrm {PGU}}}
\newcommand{\PSU}{\mathord{\mathrm {PSU}}}
\newcommand{\PGL}{\mathord{\mathrm {PGL}}}
\newcommand{\OG}{\mathord{\mathrm {O}}}

\newcommand{\id}{\mathord{\mathrm {id}}}

\newcommand{\Ker}{\operatorname{\mathrm {Ker}}\nolimits}

\newcommand{\Aut}{\operatorname{\mathrm {Aut}}\nolimits}

\newcommand{\pr}{\mathord{\mathrm {pr}}}

\newcommand{\rank}{\operatorname{\mathrm {rank}}\nolimits}

\newcommand{\closure}[1]{\overline{#1}}

\newcommand{\rmand}{\textrm{and}}

\newcommand{\quand}{\quad\rmand\quad}

\newcommand{\Tr}{\mathord{\rm Tr}}
\newcommand{\Pic}{\mathord{\rm Pic}}

\newcommand{\KerE}[1]{E_{#1}}

\newcommand{\Nef}{\mathord{\mathrm{NC}}}
\newcommand{\intM}[2]{\langle{#1}\rangle_{#2}}
\newcommand{\aut}{\Aut}

\newcommand{\sphyp}{\sp*}

\newcommand{\DR}{\mathord{{\rm DR}}}

\newcommand{\End}{\mathord{\rm End}}
\newcommand{\Km}{\mathord{\rm Km}}
\newcommand{\tilA}{\tilde{A}}
\newcommand{\SA}{S_A}
\newcommand{\SY}{S_Y}
\newcommand{\StilA}{S_{\tilA}}
\newcommand{\Fr}{\mathord{\rm Fr}}
\newcommand{\sGamma}{\Gamma}
\newcommand{\period}{\KKK}

\newcommand{\Gram}{\mathord{\rm G}}
\newcommand{\Walls}{\WWW}
\newcommand{\sqrttwo}{\sqrt{2}}
\newcommand{\Wgr}{W^{(-2)}}
\newcommand{\shrnk}{\hskip -6pt}
\newcommand{\nalmed}{\noalign{\vskip 2pt}}

\begin{document}

\title[Supersingular $K3$ surface in characteristic $5$]
{On the supersingular $K3$ surface in characteristic $5$ with Artin invariant $1$}

\author{Toshiyuki Katsura}
\author{Shigeyuki Kondo}
\author{Ichiro Shimada}

\address{T.\ Katsura:
Faculty of Science and Engineering, 
Hosei University, 
Koganei-shi, 
Tokyo,
184-8584 JAPAN
}
\email{
toshiyuki.katsura.tk@hosei.ac.jp
}
\address{S.\ Kondo:
Graduate School of Mathematics, 
Nagoya University, 
Nagoya,
 464-8602 JAPAN
}
\email{
 kondo@math.nagoya-u.ac.jp
}

\address{I.\ Shimada:
Department of Mathematics, 
Graduate School of Science, 
Hiroshima University,
1-3-1 Kagamiyama, 
Higashi-Hiroshima, 
739-8526 JAPAN
}
\email{
shimada@math.sci.hiroshima-u.ac.jp
}

\thanks{Partially supported by
 JSPS Grants-in-Aid for Scientific Research (C) No. 24540053, (S) No.22224001, (C) No.25400042.
}

\begin{abstract}
We  present  three interesting projective models of the supersingular 
$K3$ surface $X$ in characteristic $5$ with Artin invariant $1$.  
For each projective model, 
we  determine
smooth rational curves on $X$ with the  minimal degree  and the projective automorphism group. 
Moreover, by using the superspecial
abelian surface, we construct six sets of $16$ disjoint smooth rational curves on $X$,
and show that they form a beautiful configuration.
%and show that they make various configurations with a beautiful symmetry.
\end{abstract}


\maketitle

\section{Introduction}\label{sec:Intro}
Let $Y$ be a $K3$ surface defined over an algebraically closed field $k$, and
$\rho(Y)$ the Picard number of $Y$.  
Then it is well-known that $1\leq \rho(Y) \leq 20$ or $\rho(Y)=22$.  
The last case $\rho(Y)=22$ 
occurs only when $k$ is of positive characteristic.
A $K3$ surface is called \emph{supersingular} if its Picard number is $22$.
Let $Y$ be a supersingular $K3$ surface in characteristic $p\ge 3$.
Let $S_Y$ denote its N\'eron-Severi lattice and let $S_Y\dual$ be the dual of $S_Y$.
Then Artin~\cite{MR0371899} proved that 
$S_Y\dual/S_Y$ 
is a $p$-elementary abelian group of rank $2\sigma$, 
where $\sigma$ is an integer such that $1\le \sigma\le 10$.
This integer $\sigma$ is called the \emph{Artin invariant} of $Y$.
It is known that the isomorphism class of $S_Y$ depends only  on $p$ and $\sigma$
(Rudakov~and~Shafarevich~\cite{MR633161}).  
On the other hand, 
supersingular $K3$ surfaces with Artin invariant $\sigma$ form a $(\sigma -1)$-dimensional
family and a supersingular $K3$ surface with Artin invariant $1$ 
in characteristic $p$ 
is unique up to isomorphisms
(Ogus~\cite{MR563467, MR717616}, Rudakov~and~Shafarevich~\cite{MR633161}).

Supersingular $K3$ surfaces in \emph{small} characteristic $p$ with Artin invariant $1$ are especially interesting %because of their big finite automorphism groups
because big finite groups act on them %as automorphisms.
by automorphisms. 
(See~Dolgachev~and~Keum~\cite{MR2480606}).
For example, the group $\PGL(3,\F_4)\ltimes \Z/2\Z$ in case $p=2$ or $\PGU(4,\F_9)$ in case $p=3$ acts on the $K3$ surface %as automorphisms.  
by automorphisms. 
Moreover these $K3$ surfaces 
contain a finite set of smooth rational curves on which the above group acts as
symmetries.  For example, in case $p=2$, there exist $42$ smooth rational curves which form a $(21_5)$-configuration 
(see Dolgachev~and~Kondo~\cite{MR1935564K}, Katsura~and~Kondo~\cite{MR2987663}).  
In case $p=3$, 
the Fermat quartic surface is a supersingular $K3$ surface with Artin invariant $1$, 
and it contains $112$ lines (e.g. Katsura~and~Kondo~\cite{MR2862188}, Kondo and Shimada~\cite{KondoShimadaS}).

In this paper we consider a similar problem for the supersingular $K3$ surface
in characteristic $5$ with Artin invariant $1$.
We work over an algebraically closed field $k$ of characteristic $5$
containing the finite field $\F_{25}=\F_5(\sqrt{2})$.
Let $C_F$ be the Fermat sextic curve in $\P^2$ defined by 
\begin{equation}\label{Fermatsextic}
x^6 + y^6+z^6=0.
\end{equation}
Note that the left hand side of the equation (\ref{Fermatsextic}) 
is a Hermitian form over $\F_{25}$ and the projective unitary group $\PGU(3, \F_{25})$ acts on $C_F$
%as automorphisms.   
by automorphisms.
Let $\pi_F: X\to \P^2$ be the double cover of 
$\P^2$ branched along $C_F$.  
Then $X$ is a supersingular $K3$ surface
in characteristic $5$ with Artin invariant $1$, on which the finite group $\PGU(3,\F_{25})\ltimes \Z/2\Z$ acts
%as automorphisms 
by automorphisms
(e.g. Dolgachev~and~Keum~\cite{MR2480606}).  
Let $P$ be an $\F_{25}$-rational point of $C_F$.  
Then the tangent line
$\ell_P$ to $C_F$ at $P$ intersects $C_F$ at $P$ with multiplicity 6.  Hence the pullback of $\ell_P$ on
$X$ splits into two smooth rational curves meeting at one point with multiplicity $3$.
Since the number of $\F_{25}$-rational points of $C_F$ is $126$, we obtain $252$ smooth rational curves on $X$.

The main result of this paper is to exhibit three projective models of $X$ and determine
smooth rational curves of minimal degree on $X$ with respect to the corresponding polarizations.
\begin{theorem}\label{thm:main}
There exist three polarizations $h_F, h_1, h_2$ of degree $2, 60, 80$ 
on $X$ satisfying the following conditions{\rm :}

\begin{itemize}
\item[(1)] The projective model $(X, h_F)$ is the double cover of %the projective plane 
$\P^2$ branched along %the Fermat sextic curve 
$C_F$.
Here $h_F\in S_X$ is the class of the pull-back of a line on $\P^2$ by the covering 
morphism $\pi_F: X\to \P^2$.
The projective automorphism group $\Aut(X, h_F)$ of $(X, h_F)$ is
a central extension of %the automorphism group  
$\PGU(3, \F_{25})$ %of the plane curve $C_F$
by the cyclic group of order $2$ generated by the deck-transformation of $X$ over $\P^2$.
%In particular, $\Aut(X, h_F)$ is of order $756000$.
The double plane $(X, h_F)$ 
 contains exactly $252$ smooth rational curves of degree $1$, on which $\Aut(X, h_F)$ acts transitively.

\item[(2)] The projective automorphism group of $(X, h_1)$ is isomorphic to the alternating group $\mathfrak{A}_8$.
The minimal degree of curves on $(X, h_1)$ is $5$,
and $(X, h_1)$ contains exactly $168$ smooth rational curves of degree $5$,
on which $\Aut(X, h_1)$ acts transitively.

\item[(3)] The projective automorphism group of $(X, h_2)$ is isomorphic to 
$$
(\Z/2\Z)^4\rtimes (\Z/3\Z \times \mathfrak{S}_4)
$$
of order $1152$.
The minimal degree of curves on $(X, h_2)$ is $5$,
and $(X, h_2)$ contains exactly $96$ smooth rational curves of degree $5$,
which decompose into two orbits under the action of 
$\Aut(X, h_2)$.
\end{itemize}
\end{theorem}

The model $(X,h_F)$ has been  known as mentioned above.  
However we  give  another
proof of the existence of such a polarization $h_F$ on $X$ by using
the Borcherds method~\cite{MR913200, MR1654763} and a geometry of the Leech lattice.

The set of the $96$ smooth rational curves in Theorem~\ref{thm:main}~(3) possesses the following remarkable property.
%\par
%\medskip
Let $\SSS$ and $\SSS\sprime$ be two sets of  disjoint $16$ smooth rational curves on a $K3$ surface.
We say that $\SSS$ and $\SSS\sprime$  form a $(16_r)$-configuration
if every member in one set intersects  exactly $r$ members in the other set with multiplicity $1$
and is disjoint from the remaining $16-r$ members.
For example, a $(16_6)$-configuration appears in the theory of Kummer surfaces associated to 
the Jacobian of a smooth curve of genus two:  sixteen 2-torsion points on the Jacobian, the theta divisor and its translations  by 2-torsion points (Chapter 6 of Griffiths~and~Harris~\cite{MR1288523}).
\begin{theorem}\label{thm:six}
There exist six sets
$$
\SSS_{00}, \SSS_{01}, \SSS_{02}, \SSS_{10}, \SSS_{11}, \SSS_{12}
$$
of disjoint $16$ smooth rational curves on $X$
with the following properties.
\begin{itemize}
\item[(a)] If $i\ne j$,
then $\SSS_{\nu i}$ and $\SSS_{\nu j}$ form a $(16_6)$-configuration for $\nu=0$ and $1$.
\item[(b)] For $i=0,1,2$, the sets $\SSS_{0i}$ and $\SSS_{1i}$ form a $(16_{12})$-configuration.
\item[(c)] If $i\ne j$,
then $\SSS_{0 i}$ and $\SSS_{1 j}$ form a $(16_4)$-configuration.
\end{itemize}
\end{theorem}
In fact, the set of the $96$ smooth rational curves  of degree $5$ on $(X, h_2)$
decomposes into the disjoint union of six sets 
 with the properties (a), (b), (c).
 
Since $h_2^2=80$, however,
it is difficult to present these curves explicitly.
%\par
%\medskip
Instead, we construct the six sets with the properties (a), (b), (c)
on the Kummer surface model of $X$.
Let $E$ be the elliptic curve defined by $y^2=x^3-1$,
and let $A$ be the product abelian surface $E\times E$.
It is well-known that $X$ is isomorphic to the Kummer surface $\Km(A)$
associated with $A$.
In~Section~\ref{sec:Kum},
we construct these six sets  explicitly on $\Km(A)$
by giving the pull-back of rational curves by the rational map $A\cdot\cdot\to \Km(A)$.  
As a corollary of this construction, we have the following result.
Let $\P^1$ be a projective line over $\F_{25}$ with an affine parameter.
We define four subsets of $\P^1(\F_{25})$ by the following:
\begin{eqnarray*}
P_6 & =&\{\infty,0,1,2,3,4\},\\
P_4 & =&\{\sqrt {2}, \; 1+2\,\sqrt {2}, \;  3+3\,\sqrt {2}, \;  4+4\,\sqrt {2} \} ,\\
\bar{P}_4 & =&\{4\,\sqrt {2}, \;  1+3\,\sqrt {2}, \;  3+2\,\sqrt {2}, \;  4+\sqrt{2}  \},\\
P_{12}& =&\P^1(\F_{25})\setminus  (P_6 \cup  P_4 \cup  \bar{P}_4).
\end{eqnarray*}
They are mutually disjoint. See Remark~\ref{rem:psis} for the geometric characterization of 
the decomposition $\P^1(\F_{25})=P_6 \cup P_4 \cup \bar{P}_4 \cup P_{12}$.
\begin{theorem}\label{cor:six}
There exists a model of $\Km(A)$ defined over $\F_{25}$,
and a set $\SSS$ of the $96$ rational curves defined over $\F_{25}$ on $\Km(A)$
that admits a decomposition into  disjoint six subsets
$\SSS_{\nu i}$ $(\nu=0,1$ and $i=0,1,2)$
satisfying $(a), (b), (c)$ of Theorem~\ref{thm:six}.
Moreover, 
any intersection point of two curves in $\SSS$ is an $\F_{25}$-rational point, 
and, 
 for each $\varGamma$ in $\SSS_{\nu i}$,
the set $\varGamma(\F_{25})$ of $\F_{25}$-rational points on $\varGamma$  are decomposed
into the union of
disjoint four sets $\varGamma_{\nu}$, $\varGamma_{\mu i}$, $\varGamma_{\mu j}$ and $\varGamma_{\mu k}$ $($$\mu \ne\nu$, 
and  $j\ne k\ne i\ne j)$
with the following properties.
\begin{itemize}
\item[(i)] $| \varGamma_{\nu} | = 6$, $|\varGamma_{\mu i}|=12$, $|\varGamma_{\mu j}| = |\varGamma_{\mu k}|=4$.
\item[(ii)] For any point $p$ in $\varGamma_{\nu}$ and each $i\sprime \ne i$, there exists exactly one curve in $\SSS_{\nu i\sprime}$ passing through $p$.  For any point $p\sprime$ in $\varGamma_{\mu i}$,
there exists exactly 
one curve in $\SSS_{\mu i}$  passing through $p\sprime$.  
For any point $p\spprime$
in $\varGamma_{\mu j}$ $($resp.~$\varGamma_{\mu k})$,
 there exists exactly one curve in $\SSS_{\mu j}$ 
(resp.~$\SSS_{\mu k}$) passing through $p\spprime$.
\item[(iii)] 
There exists an isomorphism 
$\varphi: \varGamma\isom \P^1$
defined over $\F_{25}$
such that $\varphi\inv(P_6)=\varGamma_{\nu}$, $\varphi\inv(P_{12})=\varGamma_{\mu i}$, 
$\varphi\inv(P_4)=\varGamma_{\mu j}$ and 
$\varphi\inv(\bar{P}_4)=\varGamma_{\mu k}$.
\end{itemize}
\end{theorem}
We give three different proofs of the existence of the $96$ smooth 
rational curves mentioned in 
Theorem~\ref{thm:six}.  
We do not know whether such sets of $96$ curves
coincide under the action of the group of automorphisms of $X$.
\par
By using the Borcherds method~\cite{MR913200, MR1654763}, the groups of automorphisms of some $K3$ surfaces were calculated
(Kondo~\cite{MR1618132}, Keum~and~Kondo~\cite{MR1806732_2}, Dolgachev~and~Kondo~\cite{MR1935564K},  
Kondo and~Shimada~\cite{KondoShimadaS},
Ujigawa~\cite{UjikawaPreprint}).  
In all cases, the N\'eron-Severi lattice of each $K3$ surface is isomorphic to 
the orthogonal complement of a root lattice in $L$, where $L$ is an even unimodular lattice of signature $(1,25)$.
See Lemma~5.1 of~\cite{MR913200}, in which Borcherds gave 
a sufficient condition for the restrictions of standard fundamental domains of the reflection group of $L$ 
to the positive cone of the $K3$ surface 
to be conjugate to each other under the action of the orthogonal group of the N\'eron-Severi lattice.
%By Lemma~5.1 of~\cite{MR913200}, 
%the restrictions of fundamental domains of the reflection group of $L$ to the positive cone of the $K3$ surface are
%conjugate to each other under the action of the orthogonal group of the N\'eron-Severi lattice.
Contrary to these cases, 
a new phenomenon occurs in the present case of the supersingular $K3$ surface in characteristic $5$ 
with Artin invariant $1$:
there exist at least three non-conjugate chambers obtained by the restriction of fundamental domains
(see also~Section~\ref{subsec:further}).
The projective models in
Theorem~\ref{thm:main} correspond to these three non-conjugate chambers.  
This phenomenon also happens in the case of the complex Fermat quartic surface.

\par
\medskip

The plan of this paper is as follows.  In Section 2, we recall some lattice theory which will be used in this paper.  
Section 3 is devoted to the explanation of the  Borcherds method for finding a finite polyhedron 
in the positive cone of a hyperbolic lattice primitively embedded into the even unimodular
lattice $L$ of signature $(1,25)$.  In Section 4, we apply this method to
the case of the supersingular $K3$ surface in characteristic $5$ with Artin invariant $1$.
In particular, by using computer, we give a proof of Theorems~\ref{thm:main}~and~\ref{thm:six}.
In Section 5, by using a geometry of Leech lattice, we give another proof of
Theorems~\ref{thm:main}~and~\ref{thm:six} without using computer.  
In Section 6, we recall some facts on the supersingular elliptic curve in characteristic $5$, and in Section 7 we investigate $\F_{p^2}$-rational points on the Kummer surface associated with the product of two supersingular elliptic curves.   Section 8 is devoted to give another proof of Theorem~\ref{thm:six} by using the Kummer surface structure of $X$.  
Moreover we study the intersection between the $96$ smooth rational curves and prove Theorem~\ref{cor:six}.

\par
\medskip
In Sections~\ref{sec:bycomputer} and~\ref{sec:Kum},
we use computer for the proof of main results.
The computational data  are presented in~\cite{ShimadaCompData}.

%%%%%%%%%%%%%%%%%%

%
\section{Lattices}\label{sec:latiice}
A \emph{$\Q$-lattice}  is  a pair $(M,\intM{\cdot , \cdot}{M})$ of a free $\Z$-module  $M$ of  finite rank and 
a non-degenerate symmetric bilinear form $\intM{\cdot , \cdot}{M} : M \times M \to \Q$. 
We omit the bilinear form $\intM{\cdot , \cdot}{M}$ or the subscript $M$ in $\intM{\cdot , \cdot}{M}$
if no confusions will occur.
If $\intM{\cdot , \cdot}{}$ takes values in $\Z$,
$M$ is called a \emph{lattice}.
For $x \in M\otimes \R $, we call $x^2  =\intM{x,x}{}$ the \emph{norm} of $x$.
%A vector in $M\tensor \R$ of norm $n$ is sometimes called an \emph{$n$-vector}.
A lattice $M$ is \emph{even} if 
$x^2\in 2\Z$ holds
for any $x\in M$. 
\par
Let $M$ be a lattice of rank $r$.
The signature of $M$ is the signature of the real quadratic space $M\tensor\R$. 
We say that $M$ is \emph{negative definite} if $M\tensor\R$ is negative definite, 
and $M$ is \emph{hyperbolic}
if the signature is $(1, r-1)$.
A \emph{Gram matrix} of $M$ is an $r\times r$ matrix with entries $\intM{e_i, e_j}{}$,
where $\{e_1, \dots, e_r\}$ is a basis of $M$.
The determinant of a Gram matrix of $M$ is called the \emph{discriminant} of $M$.
\par
Let $M$ be an even lattice, and let $M\dual  ={\Hom}(M,\Z)$ 
be naturally identified with a submodule  of  $M\otimes \Q $ with the extended symmetric bilinear form.
We call this $\Q$-lattice $M\dual$ the \emph{dual lattice} of $M$.
The \emph{discriminant group} of $M$ is defined to be 
the quotient
$M\dual/M$, and is denoted by $A_M$.  
The order of $A_M$ is equal to the discriminant of $M$ up to sign.
A lattice $M$ is called \emph{unimodular} if $A_M$ is trivial,
while 
$M$ is called \emph{$p$-elementary} if $A_M$  is  $p$-elementary.  
\par
For an even lattice $M$, 
the \emph{discriminant quadratic form} of $M$ 
$$q_M : A_M \to \Q /2\Z$$ 
is defined
by $q_M(x\bmod M)  = x^2 \bmod 2\Z$.
\par
A submodule $N$ of $M$ is called \emph{primitive} if $M/N$ is torsion free.
A non-zero vector $v\in M$ is called \emph{primitive}
if the submodule of $M$ generated by  $v$ is primitive.

Let ${\OG}(M)$ be  the orthogonal group of a lattice $M$; that is, 
the group of isomorphisms of $M$ preserving $\intM{\cdot, \cdot}{}$.
We assume that ${\OG}(M)$ acts on $M$ from the \emph{right}, and 
the action of $g\in \OG(M)$ on $v\in M\tensor\R$ is denoted by $v\mapsto v^g$.
Similarly ${\OG}(q_M)$ denotes the group of isomorphisms of $A_M$ preserving $q_M$.
There is a natural homomorphism ${\OG}(M) \to {\OG}(q_M)$.
%
%\begin{equation}\label{nat}
%{\OG}(M) \to {\OG}(q_M).
%\end{equation}
%
%whose kernel is denoted by ${\OG}(M)\dual$.

Let $M$ be a hyperbolic lattice.
A \emph{positive cone} of $M$ is one of the two connected components of the set
$$
\set{x\in M\tensor \R}{x^2>0}.
$$
Let $\PPP_{M}$ be a positive cone of $M$.
We denote by ${\OG}^+(M)$ the group of isometries of $M$ preserving  $\PPP_{M}$.
Then ${\OG}(M)={\OG}^+(M)\times\{\pm 1\}$.
For a vector $v\in M\tensor \R$ with $v^2<0$,  we define 
$$
(v)\sperp =\set{x\in \PPP_M}{\intM{x, v}{}=0},
$$
which is   a real hyperplane of $\PPP_M$.
An isometry $g\in \OG^+(M)$  is called a \emph{reflection with respect to $v$}
or a \emph{reflection into  $(v)\sperp$}
if $g$ is of order $2$ and fixes each point of $(v)\sperp$.
 For a lattice $M$,
 the set of $(-2)$-vectors is denoted by $\RRR_M$.
Any element $r$ of $\RRR_M$ defines a reflection
$$
s_r : x\mapsto x+\intM{x, r}{} r
$$
with respect to $r$.  We denote by $\Wgr (M)$ the group generated by  the set of reflections 
$\shortset{s_r}{r\in \RRR_M}$.
Since $s_r$ preserves $\PPP_M$, $\Wgr (M)$ is a subgroup of $\OG^+(M)$ . 
It is obvious that $\Wgr (M)$ is normal in $\OG^+(M)$.
\par
A negative definite even lattice $M$ is said to be a \emph{root lattice}
if $M$ is generated by $\RRR_M$.
%
%
%%%%%%%%%%%%%%%%%%

%
\section{Borcherds method}\label{GBM}
In this section, we review the Borcherds method~\cite{MR913200, MR1654763},
and the algorithms in~\cite{ShimadaAlgoAut}.
\par
\medskip
We define some notions and fix some notation.
Let $M$ be an even hyperbolic lattice with a fixed positive cone $\PPP_M$.
%For a vector $v\in M\tensor\R$ with $v^2<0$, we denote by $(v)\sperp$
%the hyperplane $\shortset{x\in \PPP_M}{\intM{x, v}{M}=0}$
%of $\PPP_M$.
Let   $\VVV$ be a set of vectors $v\in M\tensor\R$ with $v^2<0$.
Suppose that the family of hyperplanes 
$$
\VVV\sphyp =\set{(v)\sperp}{v\in \VVV}
$$
is locally finite in $\PPP_M$.
By a  \emph{$\VVV\sphyp$-chamber}, we mean a closure in $\PPP_M$ of a connected component of 
$$
\PPP_M\setminus \bigcup_{ v \in \VVV} (v)\sperp.
$$
Let $D$ be a \emph{$\VVV\sphyp$}-chamber.
A hyperplane $(v)\sperp$ is said to be a \emph{wall} of $D$ if 
$(v)\sperp$ is disjoint from the interior of $D$ and $(v)\sperp\cap D$ contains a non-empty open subset of $(v)\sperp$.
\par
Recall that $\RRR_M$ is the set of vectors $r\in M$ with $r^2=-2$.
Then each $\RRR_M\sphyp$-chamber is a fundamental domain of the action  of $\Wgr (M)$ on $\PPP_M$.
\subsection{Conway-Borcherds  theory}\label{subsec:Conway}
Let $L$ be an even unimodular hyperbolic lattice of rank $26$,
which is unique up to isomorphisms, and 
let $\PPP_L$ be a positive cone of $L$.
An $\RRR_L\sphyp$-chamber will be called a \emph{Conway chamber}.
A non-zero primitive vector $w\in L$ with $w^2=0$ is called a \emph{Weyl vector}
if $w$ is contained in the closure $\closure{\PPP}_L$ of $\PPP_L$ in 
$L\tensor \R$ and 
the even negative-definite unimodular lattice $\gen{w}\sperp/\gen{w}$ is isomorphic to the 
(negative-definite) Leech lattice 
%$\Lambda$
(that is, $\gen{w}\sperp/\gen{w}$ contains no $(-2)$-vectors).
For a Weyl vector $w$,
we put 
\begin{equation}\label{eq:defDelta}
\Delta(w) =\set{r\in \RRR_L}{ \intM{r, w}{}=1}.
\end{equation}
Conway and Sloane~\cite{MR640949} and 
Conway~\cite{MR690711} proved the following:
\begin{theorem}
If $w$ is a Weyl vector, then
$$
\DDD(w) =\set{x\in \PPP_L}{\intM{r, x}{}\ge 0\;\;\textrm{for any}\;\; r\in \Delta(w)}
$$
is a Conway chamber,
and $\shortset{(r)\sperp}{ r\in \Delta(w)}$ is the set of walls of $\DDD(w)$.
For any Conway chamber $\DDD$, there exists a unique Weyl vector $w$ such that $\DDD=\DDD(w)$.
\end{theorem}
Let $S$ be an even hyperbolic  lattice of rank $<26$.
Suppose that $S$ is primitively embedded into $L$.
Let $\PPP_S$ be the positive cone of $S$ that is contained in $\PPP_L$.
Let $R$ denote the orthogonal complement of $S$ in $L$. 
For $x\in L\tensor \R$, we denote by
$$
x\mapsto x_S\quand x\mapsto x_R, 
$$
the projections to $S\tensor \R$ and $R\tensor \R$, respectively.
Note that, if $v\in L$, then $v_S\in S\dual$ and $v_R\in R\dual$.
We assume the following:
\begin{itemize}
\item[(i)]
The negative-definite lattice $R$ cannot be embedded into the Leech lattice.
(For example,   this condition is satisfied if $\RRR_R\ne \emptyset$.)
\item[(ii)]
The natural homomorphism  $\OG(R)\to \OG(q_R)$
 is surjective.
\end{itemize}
We put
$$
\RRR_{L|S} =\set{r_S}{r\in \RRR_L, \;\; \intM{r_S, r_S}{}<0}.
$$
It is easy to see that 
the family of hyperplanes $\RRR_{L|S}\sphyp$ is locally finite in $\PPP_S$.
A Conway chamber $\DDD$ is said to be \emph{$S$-nondegenerate} if $\DDD\cap \PPP_S$ contains 
a non-empty open subset of $\PPP_S$.
If $\DDD$ is an $S$-nondegenerate Conway chamber,
then $D =\DDD\cap \PPP_S$ is an $\RRR_{L|S}\sphyp$-chamber of $\PPP_S$,
which is called an \emph{induced chamber}.
Since $\PPP_L$ is tessellated by Conway chambers,
$\PPP_S$ is tessellated by induced chambers.
%Since $\RRR_S\sphyp$ is a subset of $\RRR_{L|S}\sphyp$,
Since $\RRR_S$ is a subset of $\RRR_{L|S}$,
any $\RRR_S\sphyp$-chamber is a union of induced chambers.
We have the following. See~\cite{ShimadaAlgoAut}.
\begin{proposition}
{\rm (1)} Any  induced chamber has only a finite number of walls.
%If $D =\DDD(w)\cap \PPP_S$ is an induced chamber, then 
%and the walls of $D$ are contained in $\shortset{(r)\sperp}{r\in \Delta(w)}$.

{\rm (2)} The automorphism group 
$\aut(D) =\shortset{g\in \OG^+(S)}{D^g=D}$ 
of an induced chamber $D$ is a finite group.
\end{proposition}
In~\cite{ShimadaAlgoAut},
we have presented algorithms to calculate the set of walls and the automorphism group 
of an induced chamber.
Moreover, by an algorithm in~\cite{ShimadaAlgoAut},
if we have
\begin{itemize}
\item a Weyl vector $w\in L$ such that $\DDD(w)$ is $S$-nondegenerate,
and 
\item a wall $(v)\sperp$ of the induced chamber $D =\DDD(w)\cap \PPP_S$,
\end{itemize}
then we can calculate a Weyl vector $w\sprime \in L$ such that $D\sprime =\DDD(w\sprime)\cap \PPP_S$ is
the induced chamber adjacent to $D$ along the wall $(v)\sperp$.
\subsection{Periods and automorphisms  of supersingular $K3$ surfaces}\label{subsec:period}
Let $Y$ be a supersingular $K3$ surface defined over an algebraically closed field $k$ of odd characteristic $p$
with Artin invariant $\sigma$,
and let $S_Y$ denote the N\'eron-Severi lattice of $Y$.
Since $S_Y\dual/S_Y$ is $p$-elementary, we have  $pS_Y\dual \subset S_Y$.
Consider the $2\sigma$-dimensional $\F_p$-vector space
$$
S_0 =p S_Y\dual/ p S_Y \;\;\subset\;\; S_Y\tensor_{\Z}\F_p,
$$
on which we have an $\F_p$-valued quadratic form
$Q_0: S_0 \to \F_p$ defined by
$$
Q_0\;\;:\;\; px \bmod p{S_Y} \mapsto px^2 \bmod p \quad (x\in S_Y\dual).
$$
Let 
$\bar{c}_{\DR}: S_Y\tensor k \to H^2_{\DR}(Y)$ 
be the Chern class map.
Then $\Ker (\bar{c}_{\DR})$ is a $\sigma$-dimensional  isotropic subspace of $Q_0\tensor k$.
Let $\varphi : S_0\tensor k \to S_0\tensor k$ denote the map $\id\tensor F_k$,
where $F_k$ is the Frobenius of $k$.
\begin{definition}
The \emph{period} $\period_Y$ of $Y$ is defined to be $\varphi^*(\Ker (\bar{c}_{\DR}))$.
\end{definition}
Note that $\OG(S_Y)$ acts on $(S_0, Q_0)$ naturally.
We put 
 $$
 G_Y =\set{g\in \OG(S_Y)}{\period_Y^g=\period_Y}.
$$
We denote by $\PPP_{S_Y}$ the positive cone of $S_Y$ containing an ample class of $Y$.
Let $\Nef(Y)$ denote the intersection of $\PPP_{S_Y}$ with the nef cone of $Y$;
$$
\Nef(Y) =\set{x\in \PPP_{S_Y}}{\intM{x, C}{}\ge 0\;\;\textrm{for any curve $C$ on $Y$}\;\;}.
$$
We put
$$
 \aut(\Nef(Y)) =\set{g\in \OG^+(S_Y)}{\Nef(Y)^g=\Nef(Y)}.
$$
Thanks to  the Torelli theorem for supersingular $K3$ surfaces in odd characteristics due to Ogus~\cite{MR563467, MR717616},
we see that the natural action of $\Aut(Y)$ on $S_Y$ identifies $\Aut(Y)$ with
$$
\aut(\Nef(Y))\cap G_Y.
$$
\par
\medskip
Now suppose that $S_Y$ is embedded into $L$
in such a way that the conditions (i) and (ii) in Section~\ref{subsec:Conway}
are satisfied
and that the image of $\Nef(Y)$ is contained in $\PPP_L$.
It is well-known that  $\Nef(Y)$ is an $\RRR_{S_Y}\sphyp$-chamber in $\PPP_{S_Y}$.
(See, for example, Rudakov and Shafarevich~\cite{MR633161}.) 
Hence $\Nef(Y)$ is tessellated by induced chambers.
For an induced chamber $D$ contained in $\Nef(Y)$, we put
$$
\aut_{Y}(D) =\aut(D)\cap G_Y.
$$
Then $\aut_{Y}(D)$ is a finite subgroup of $\Aut(Y)=\aut(\Nef(Y))\cap G_Y$.
More precisely,
if $v\in D\cap S_Y$ is a vector in the interior of $D$,
then
$$
h_D =\sum_{g\in \aut_{Y}(D)} v^g
$$
is an ample class,
and $\aut_{Y}(D)$ is the  automorphism group $\Aut(Y, h_D)$ 
of the polarized $K3$ surface $(Y, h_D)$.
We have an algorithm
to make the complete list of elements of $\aut(D)$.
Hence, in order to calculate $\Aut(Y, h_D)$,
all we have  to do is to calculate the action of $\OG(S_Y)$ on the period $\KKK_Y$.
\par
\medskip
We say that two induced chambers $D$ and $D\sprime$ are  \emph{$G_Y$-congruent} 
if there exists $g\in G_Y$
such that $D^g=D\sprime$.
The number of $G_Y$-congruence classes is finite. 
If we obtain the list of all $G_Y$-congruence classes,
we can  determine the automorphism group of $Y$.
(As is explained in Introduction,
in the previous works of computing automorphism groups of $K3$ surfaces
using this technique, 
 there exists only one $O^+(S_Y)$-congruence class.)
See~\cite{ShimadaAlgoAut} and Section~\ref{subsec:further}. 

%
%
%%%%%%%%%%%%%%%%%%
%
%%%%%%%%%%%%%%%%%%
%
\section{Proof of Theorems by computer }\label{sec:bycomputer}
In this section and the next, we  prove Theorems~\ref{thm:main} and~\ref{thm:six} by calculating
some induced chambers.  
In this section we give a proof based on the algorithm presented in~\cite{ShimadaAlgoAut}.
\subsection{The N\'eron-Severi lattice and the period of $X$}\label{sec:SX}
Using the projective model $(X, h_F)$,
we calculate the N\'eron-Severi lattice $S_X$  and the period $\period_X$ of $X$
explicitly.
\par
\medskip
As is explained in Introduction, 
the surface $X$ contains $252$ smooth rational curves $\varGamma$ such that
$\intM{\varGamma, h_F}{}=1$.
We call these smooth rational curves \emph{$h_F$-lines}.
\begin{table}
{\small
$$
\begin{array}{lll}
 P_{1}:=[0: 1: 2] & 
 P_{2}:=[0: 1: 3] & 
 P_{3}:=[0: 1: 1+\sqrttwo]\\
 P_{4}:=[0: 1: 4+\sqrttwo] & 
 P_{5}:=[0: 1: 1+4\sqrttwo] & 
 P_{6}:=[0: 1: 4+4\sqrttwo]\\
 P_{7}:=[1: 0: 2] & 
 P_{8}:=[1: 0: 3] & 
 P_{9}:=[1: 0: 1+\sqrttwo]\\
 P_{10}:=[1: 0: 4+\sqrttwo] & 
 P_{11}:=[1: 0: 1+4\sqrttwo] & 
 P_{12}:=[1: 0: 4+4\sqrttwo]\\
 P_{13}:=[1: 1: \sqrttwo] & 
 P_{14}:=[1: 1: 1+2\sqrttwo] & 
 P_{15}:=[1: 1: 4+2\sqrttwo]\\
 P_{16}:=[1: 1: 1+3\sqrttwo] & 
 P_{17}:=[1: 1: 4+3\sqrttwo] & 
 P_{18}:=[1: 1: 4\sqrttwo]\\
 P_{19}:=[1: 2: 0] & 
 P_{20}:=[1: 3: 0] & 
 P_{21}:=[1: 4: \sqrttwo]\\
 P_{22}:=[1: 4: 1+2\sqrttwo] & 
 P_{23}:=[1: 4: 4+2\sqrttwo] & 
 P_{24}:=[1: 4: 1+3\sqrttwo]\\
 P_{25}:=[1: 4: 4+3\sqrttwo] & 
 P_{26}:=[1: 4: 4\sqrttwo] & 
 P_{27}:=[1: \sqrttwo: 1]\\
 P_{28}:=[1: \sqrttwo: 4] & 
 P_{29}:=[1: \sqrttwo: 2+2\sqrttwo] & 
 P_{30}:=[1: \sqrttwo: 3+2\sqrttwo]\\
 P_{31}:=[1: \sqrttwo: 2+3\sqrttwo] & 
 P_{32}:=[1: \sqrttwo: 3+3\sqrttwo] & 
 P_{33}:=[1: 1+\sqrttwo: 0]\\
 P_{34}:=[1: 2+\sqrttwo: 2+\sqrttwo] & 
 P_{35}:=[1: 2+\sqrttwo: 3+\sqrttwo] & 
 P_{36}:=[1: 2+\sqrttwo: 2\sqrttwo]\\
 P_{37}:=[1: 2+\sqrttwo: 3\sqrttwo] & 
 P_{38}:=[1: 2+\sqrttwo: 2+4\sqrttwo] & 
 P_{39}:=[1: 2+\sqrttwo: 3+4\sqrttwo]\\
\dots &\dots& \dots \\
\dots &\dots& \dots \\
 P_{124}:=[1: 3+4\sqrttwo: 2+4\sqrttwo] & 
 P_{125}:=[1: 3+4\sqrttwo: 3+4\sqrttwo] & 
 P_{126}:=[1: 4+4\sqrttwo: 0]\\
\end{array}
$$

}
\caption{$\F_{25}$-rational points on $C_F$}\label{table:F25FS}
\end{table}
The $h_F$-lines are labelled as follows.
Let $\pi_F: X\to \P^2$ denote the double covering.
Part of 
the $\F_{25}$-rational points 
$P_1, \dots, P_{126}$ on the Fermat curve $C_F$ of degree $6$ are given explicitly in Table~\ref{table:F25FS}.
Let $l_i$ be the line on $\P^2$ tangent to $C_F$ at $P_i$.
We put
$$
l_1^+ =\{w=x^3, y=3z\}\subset X,
$$
which is an irreducible component of $\pi_F^*(l_1)$,
and let $l_1^-$ denote the other irreducible component.
For $i>1$, let $l_i^+$ be the irreducible component of $\pi_F^*(l_i)$
such that $\intM{[l_1^+], [l_i^+]}{}=1$,
and let $l_i^-$ be the other irreducible component.
Consider the following twenty-two $h_F$-lines.
\begin{eqnarray*}
&&
\ell_{1} =l_{1}^+, \;\;\ell_{2} =l_{1}^-, \;\;\ell_{3} =l_{2}^+, \;\;\ell_{4} =l_{3}^+, \;\;\ell_{5} =l_{4}^+, \;\;\ell_{6} =l_{5}^+, \;\;\ell_{7} =l_{7}^+, \;\;\ell_{8} =l_{8}^+, \;\;\\
&&
\ell_{9} =l_{9}^+, \;\;\ell_{10} =l_{10}^+, \;\;\ell_{11} =l_{13}^+, \;\;\ell_{12} =l_{14}^+, \;\;\ell_{13} =l_{15}^+, \;\;\ell_{14} =l_{16}^+, \;\;\ell_{15} =l_{17}^+, \;\;\\
&&
\ell_{16} =l_{19}^+, \;\;\ell_{17} =l_{21}^+, \;\;\ell_{18} =l_{22}^+, \;\;\ell_{19} =l_{24}^+, \;\;\ell_{20} =l_{25}^+, \;\;\ell_{21} =l_{27}^+, \;\;\ell_{22} =l_{34}^+.
\end{eqnarray*}
Their intersection matrix is of determinant $-25$.
Hence the classes of these $h_F$-lines form a basis of $S_X$.
The Gram matrix $\Gram_{S}$ of $S_X$ with respect to this basis 
 $[\ell_1], \dots, [\ell_{22}]$
 is given in Table~\ref{table:GramNS}.
\begin{table}
%
% Here is M_NS
%
{\small
$$
\left[ \begin {array}{cccccccccccccccccccccc}
-2\shrnk &3\shrnk &1\shrnk &1\shrnk &1\shrnk &1\shrnk &1\shrnk &1\shrnk &1\shrnk &1\shrnk &1\shrnk &1\shrnk &1\shrnk &1\shrnk &1\shrnk &1\shrnk &1\shrnk &1\shrnk &1\shrnk &1\shrnk &1\shrnk &1\\
\nalmed
3\shrnk &-2\shrnk &0\shrnk &0\shrnk &0\shrnk &0\shrnk &0\shrnk &0\shrnk &0\shrnk &0\shrnk &0\shrnk &0\shrnk &0\shrnk &0\shrnk &0\shrnk &0\shrnk &0\shrnk &0\shrnk &0\shrnk &0\shrnk &0\shrnk &0\\
\nalmed
1\shrnk &0\shrnk &-2\shrnk &1\shrnk &1\shrnk &1\shrnk &0\shrnk &0\shrnk &0\shrnk &0\shrnk &0\shrnk &1\shrnk &1\shrnk &1\shrnk &1\shrnk &0\shrnk &0\shrnk &1\shrnk &1\shrnk &1\shrnk &0\shrnk &0\\
\nalmed
1\shrnk &0\shrnk &1\shrnk &-2\shrnk &1\shrnk &1\shrnk &0\shrnk &0\shrnk &0\shrnk &0\shrnk &1\shrnk &0\shrnk &1\shrnk &0\shrnk &0\shrnk &0\shrnk &0\shrnk &0\shrnk &1\shrnk &0\shrnk &1\shrnk &1\\
\nalmed
1\shrnk &0\shrnk &1\shrnk &1\shrnk &-2\shrnk &1\shrnk &1\shrnk &1\shrnk &1\shrnk &1\shrnk &0\shrnk &1\shrnk &0\shrnk &0\shrnk &0\shrnk &1\shrnk &1\shrnk &0\shrnk &0\shrnk &1\shrnk &0\shrnk &0\\
\nalmed
1\shrnk &0\shrnk &1\shrnk &1\shrnk &1\shrnk &-2\shrnk &0\shrnk &0\shrnk &0\shrnk &0\shrnk &0\shrnk &0\shrnk &0\shrnk &0\shrnk &1\shrnk &0\shrnk &1\shrnk &1\shrnk &0\shrnk &0\shrnk &0\shrnk &1\\
\nalmed
1\shrnk &0\shrnk &0\shrnk &0\shrnk &1\shrnk &0\shrnk &-2\shrnk &0\shrnk &0\shrnk &1\shrnk &0\shrnk &0\shrnk &0\shrnk &0\shrnk &0\shrnk &0\shrnk &0\shrnk &1\shrnk &1\shrnk &1\shrnk &1\shrnk &0\\
\nalmed
1\shrnk &0\shrnk &0\shrnk &0\shrnk &1\shrnk &0\shrnk &0\shrnk &-2\shrnk &1\shrnk &0\shrnk &0\shrnk &1\shrnk &1\shrnk &1\shrnk &1\shrnk &0\shrnk &0\shrnk &0\shrnk &0\shrnk &0\shrnk &1\shrnk &1\\
\nalmed
1\shrnk &0\shrnk &0\shrnk &0\shrnk &1\shrnk &0\shrnk &0\shrnk &1\shrnk &-2\shrnk &0\shrnk &1\shrnk &0\shrnk &1\shrnk &0\shrnk &0\shrnk &0\shrnk &1\shrnk &1\shrnk &1\shrnk &1\shrnk &0\shrnk &1\\
\nalmed
1\shrnk &0\shrnk &0\shrnk &0\shrnk &1\shrnk &0\shrnk &1\shrnk &0\shrnk &0\shrnk &-2\shrnk &1\shrnk &0\shrnk &1\shrnk &1\shrnk &1\shrnk &0\shrnk &1\shrnk &1\shrnk &0\shrnk &0\shrnk &0\shrnk &1\\
\nalmed
1\shrnk &0\shrnk &0\shrnk &1\shrnk &0\shrnk &0\shrnk &0\shrnk &0\shrnk &1\shrnk &1\shrnk &-2\shrnk &1\shrnk &0\shrnk &0\shrnk &1\shrnk &1\shrnk &0\shrnk &1\shrnk &1\shrnk &0\shrnk &0\shrnk &0\\
\nalmed
1\shrnk &0\shrnk &1\shrnk &0\shrnk &1\shrnk &0\shrnk &0\shrnk &1\shrnk &0\shrnk &0\shrnk &1\shrnk &-2\shrnk &0\shrnk &0\shrnk &1\shrnk &1\shrnk &0\shrnk &1\shrnk &1\shrnk &0\shrnk &1\shrnk &1\\
\nalmed
1\shrnk &0\shrnk &1\shrnk &1\shrnk &0\shrnk &0\shrnk &0\shrnk &1\shrnk &1\shrnk &1\shrnk &0\shrnk &0\shrnk &-2\shrnk &1\shrnk &0\shrnk &0\shrnk &1\shrnk &0\shrnk &0\shrnk &1\shrnk &0\shrnk &0\\
\nalmed
1\shrnk &0\shrnk &1\shrnk &0\shrnk &0\shrnk &0\shrnk &0\shrnk &1\shrnk &0\shrnk &1\shrnk &0\shrnk &0\shrnk &1\shrnk &-2\shrnk &0\shrnk &1\shrnk &0\shrnk &1\shrnk &1\shrnk &0\shrnk &0\shrnk &0\\
\nalmed
1\shrnk &0\shrnk &1\shrnk &0\shrnk &0\shrnk &1\shrnk &0\shrnk &1\shrnk &0\shrnk &1\shrnk &1\shrnk &1\shrnk &0\shrnk &0\shrnk &-2\shrnk &0\shrnk &1\shrnk &0\shrnk &0\shrnk &1\shrnk &0\shrnk &0\\
\nalmed
1\shrnk &0\shrnk &0\shrnk &0\shrnk &1\shrnk &0\shrnk &0\shrnk &0\shrnk &0\shrnk &0\shrnk &1\shrnk &1\shrnk &0\shrnk &1\shrnk &0\shrnk &-2\shrnk &1\shrnk &0\shrnk &0\shrnk &1\shrnk &0\shrnk &0\\
\nalmed
1\shrnk &0\shrnk &0\shrnk &0\shrnk &1\shrnk &1\shrnk &0\shrnk &0\shrnk &1\shrnk &1\shrnk &0\shrnk &0\shrnk &1\shrnk &0\shrnk &1\shrnk &1\shrnk &-2\shrnk &0\shrnk &1\shrnk &0\shrnk &1\shrnk &0\\
\nalmed
1\shrnk &0\shrnk &1\shrnk &0\shrnk &0\shrnk &1\shrnk &1\shrnk &0\shrnk &1\shrnk &1\shrnk &1\shrnk &1\shrnk &0\shrnk &1\shrnk &0\shrnk &0\shrnk &0\shrnk &-2\shrnk &0\shrnk &1\shrnk &1\shrnk &0\\
\nalmed
1\shrnk &0\shrnk &1\shrnk &1\shrnk &0\shrnk &0\shrnk &1\shrnk &0\shrnk &1\shrnk &0\shrnk &1\shrnk &1\shrnk &0\shrnk &1\shrnk &0\shrnk &0\shrnk &1\shrnk &0\shrnk &-2\shrnk &0\shrnk &0\shrnk &0\\
\nalmed
1\shrnk &0\shrnk &1\shrnk &0\shrnk &1\shrnk &0\shrnk &1\shrnk &0\shrnk &1\shrnk &0\shrnk &0\shrnk &0\shrnk &1\shrnk &0\shrnk &1\shrnk &1\shrnk &0\shrnk &1\shrnk &0\shrnk &-2\shrnk &1\shrnk &1\\
\nalmed
1\shrnk &0\shrnk &0\shrnk &1\shrnk &0\shrnk &0\shrnk &1\shrnk &1\shrnk &0\shrnk &0\shrnk &0\shrnk &1\shrnk &0\shrnk &0\shrnk &0\shrnk &0\shrnk &1\shrnk &1\shrnk &0\shrnk &1\shrnk &-2\shrnk &0\\
\nalmed
1\shrnk &0\shrnk &0\shrnk &1\shrnk &0\shrnk &1\shrnk &0\shrnk &1\shrnk &1\shrnk &1\shrnk &0\shrnk &1\shrnk &0\shrnk &0\shrnk &0\shrnk &0\shrnk &0\shrnk &0\shrnk &0\shrnk &1\shrnk &0\shrnk &-2\end {array} \right] 

$$
}
\caption{Gram matrix of $S_X$}\label{table:GramNS}
\end{table}
An element  of $S_X\tensor \R$ is usually written as a row vector 
$[x_1, \dots, x_{22}]$ with respect to
the basis   $[\ell_1], \dots, [\ell_{22}]$,
while 
when it is written with respect to 
the \emph{dual} basis $[\ell_1]\dual, \dots, [\ell_{22}]\dual$,
we use the notation
$[\xi_1, \dots, \xi_{22}]\dual$.
For example, we have
\begin{eqnarray*}
h_F&=&[1,1,0,0,0,0,0,0,0,0,0,0,0,0,0,0,0,0,0,0,0,0]\\
&=&[1,1,1,1,1,1,1,1,1,1,1,1,1,1,1,1,1,1,1,1,1,1]\dual,\\
{}[l_7^-]&=&[1, 1, 0, 0, 0, 0, -1, 0, 0, 0, 0, 0, 0, 0, 0, 0, 0, 0, 0, 0, 0, 0] \\
&=&[0, 1, 1, 1, 0, 1, 3, 1, 1, 0, 1, 1, 1, 1, 1, 1, 1, 0, 0, 0, 0, 1]\dual, \\
{}[l_{14}^+]&=&[0, 0, 0, 0, 0, 0, 0, 0, 0, 0, 0, 1, 0, 0, 0, 0, 0, 0, 0, 0, 0, 0]\\
&=&[1, 0, 1, 0, 1, 0, 0, 1, 0, 0, 1, -2, 0, 0, 1, 1, 0, 1, 1, 0, 1, 1]\dual.
\end{eqnarray*}
We let $\OG(S_X)$ act on $S_X$ from the \emph{right},
so that we have 
$$
\OG(S_X)=\set{g\in \GL_{22}(\Z)}{ g\cdot \Gram_{S}\cdot  {}^t g=\Gram_{S}}. 
$$
The substitution
$\sqrt{2}\mapsto -\sqrt{2}$
induces a permutation on the set of $h_F$-lines
preserving the intersection form,
and hence it induces an isometry of the lattice $S_X$,
which is given by 
 the right multiplication
of the matrix  in Table~\ref{table:Frob}.
The deck-transformation of $\pi_F: X\to \P^2$
also induces an isometry of $S_X$,
which is given by
\begin{equation}\label{eq:flip}
[\ell_1]\mapsto [\ell_2],
\quad
[\ell_2]\mapsto [\ell_1],
\quand
[\ell_i]\mapsto h_F-[\ell_i]\quad\textrm{for}\quad i>2.
\end{equation}
\begin{table}
{\small
$$
\left[\begin {array}{cccccccccccccccccccccc} 
1\shrnk &0\shrnk &0\shrnk &0\shrnk &0\shrnk &0\shrnk &0\shrnk &0\shrnk &0\shrnk &0\shrnk &0\shrnk &0\shrnk &0\shrnk &0\shrnk &0\shrnk &0\shrnk &0\shrnk &0\shrnk &0\shrnk &0\shrnk &0\shrnk &0\\
\nalmed
0\shrnk &1\shrnk &0\shrnk &0\shrnk &0\shrnk &0\shrnk &0\shrnk &0\shrnk &0\shrnk &0\shrnk &0\shrnk &0\shrnk &0\shrnk &0\shrnk &0\shrnk &0\shrnk &0\shrnk &0\shrnk &0\shrnk &0\shrnk &0\shrnk &0\\
\nalmed
0\shrnk &0\shrnk &1\shrnk &0\shrnk &0\shrnk &0\shrnk &0\shrnk &0\shrnk &0\shrnk &0\shrnk &0\shrnk &0\shrnk &0\shrnk &0\shrnk &0\shrnk &0\shrnk &0\shrnk &0\shrnk &0\shrnk &0\shrnk &0\shrnk &0\\
\nalmed
0\shrnk &0\shrnk &0\shrnk &0\shrnk &0\shrnk &1\shrnk &0\shrnk &0\shrnk &0\shrnk &0\shrnk &0\shrnk &0\shrnk &0\shrnk &0\shrnk &0\shrnk &0\shrnk &0\shrnk &0\shrnk &0\shrnk &0\shrnk &0\shrnk &0\\
\nalmed
2\shrnk &3\shrnk &-1\shrnk &-1\shrnk &-1\shrnk &-1\shrnk &0\shrnk &0\shrnk &0\shrnk &0\shrnk &0\shrnk &0\shrnk &0\shrnk &0\shrnk &0\shrnk &0\shrnk &0\shrnk &0\shrnk &0\shrnk &0\shrnk &0\shrnk &0\\
\nalmed
0\shrnk &0\shrnk &0\shrnk &1\shrnk &0\shrnk &0\shrnk &0\shrnk &0\shrnk &0\shrnk &0\shrnk &0\shrnk &0\shrnk &0\shrnk &0\shrnk &0\shrnk &0\shrnk &0\shrnk &0\shrnk &0\shrnk &0\shrnk &0\shrnk &0\\
\nalmed
0\shrnk &0\shrnk &0\shrnk &0\shrnk &0\shrnk &0\shrnk &1\shrnk &0\shrnk &0\shrnk &0\shrnk &0\shrnk &0\shrnk &0\shrnk &0\shrnk &0\shrnk &0\shrnk &0\shrnk &0\shrnk &0\shrnk &0\shrnk &0\shrnk &0\\
\nalmed
0\shrnk &0\shrnk &0\shrnk &0\shrnk &0\shrnk &0\shrnk &0\shrnk &1\shrnk &0\shrnk &0\shrnk &0\shrnk &0\shrnk &0\shrnk &0\shrnk &0\shrnk &0\shrnk &0\shrnk &0\shrnk &0\shrnk &0\shrnk &0\shrnk &0\\
\nalmed
0\shrnk &0\shrnk &1\shrnk &1\shrnk &0\shrnk &1\shrnk &0\shrnk &-1\shrnk &-1\shrnk &0\shrnk &0\shrnk &0\shrnk &0\shrnk &0\shrnk &0\shrnk &0\shrnk &0\shrnk &0\shrnk &0\shrnk &0\shrnk &0\shrnk &0\\
\nalmed
0\shrnk &0\shrnk &1\shrnk &1\shrnk &0\shrnk &1\shrnk &-1\shrnk &0\shrnk &0\shrnk &-1\shrnk &0\shrnk &0\shrnk &0\shrnk &0\shrnk &0\shrnk &0\shrnk &0\shrnk &0\shrnk &0\shrnk &0\shrnk &0\shrnk &0\\
\nalmed
0\shrnk &0\shrnk &0\shrnk &0\shrnk &0\shrnk &0\shrnk &0\shrnk &0\shrnk &0\shrnk &0\shrnk &1\shrnk &1\shrnk &-1\shrnk &-1\shrnk &1\shrnk &0\shrnk &0\shrnk &0\shrnk &0\shrnk &0\shrnk &0\shrnk &0\\
\nalmed
0\shrnk &0\shrnk &0\shrnk &0\shrnk &0\shrnk &0\shrnk &0\shrnk &0\shrnk &0\shrnk &0\shrnk &0\shrnk &0\shrnk &0\shrnk &1\shrnk &0\shrnk &0\shrnk &0\shrnk &0\shrnk &0\shrnk &0\shrnk &0\shrnk &0\\
\nalmed
0\shrnk &0\shrnk &0\shrnk &0\shrnk &0\shrnk &0\shrnk &0\shrnk &0\shrnk &0\shrnk &0\shrnk &0\shrnk &0\shrnk &0\shrnk &0\shrnk &1\shrnk &0\shrnk &0\shrnk &0\shrnk &0\shrnk &0\shrnk &0\shrnk &0\\
\nalmed
0\shrnk &0\shrnk &0\shrnk &0\shrnk &0\shrnk &0\shrnk &0\shrnk &0\shrnk &0\shrnk &0\shrnk &0\shrnk &1\shrnk &0\shrnk &0\shrnk &0\shrnk &0\shrnk &0\shrnk &0\shrnk &0\shrnk &0\shrnk &0\shrnk &0\\
\nalmed
0\shrnk &0\shrnk &0\shrnk &0\shrnk &0\shrnk &0\shrnk &0\shrnk &0\shrnk &0\shrnk &0\shrnk &0\shrnk &0\shrnk &1\shrnk &0\shrnk &0\shrnk &0\shrnk &0\shrnk &0\shrnk &0\shrnk &0\shrnk &0\shrnk &0\\
\nalmed
0\shrnk &0\shrnk &0\shrnk &0\shrnk &0\shrnk &0\shrnk &0\shrnk &0\shrnk &0\shrnk &0\shrnk &0\shrnk &0\shrnk &0\shrnk &0\shrnk &0\shrnk &1\shrnk &0\shrnk &0\shrnk &0\shrnk &0\shrnk &0\shrnk &0\\
\nalmed
0\shrnk &0\shrnk &0\shrnk &0\shrnk &0\shrnk &0\shrnk &0\shrnk &0\shrnk &0\shrnk &0\shrnk &1\shrnk &1\shrnk &-1\shrnk &0\shrnk &0\shrnk &0\shrnk &0\shrnk &0\shrnk &1\shrnk &-1\shrnk &0\shrnk &0\\
\nalmed
0\shrnk &0\shrnk &0\shrnk &0\shrnk &0\shrnk &0\shrnk &0\shrnk &0\shrnk &0\shrnk &0\shrnk &0\shrnk &0\shrnk &0\shrnk &0\shrnk &0\shrnk &0\shrnk &0\shrnk &0\shrnk &1\shrnk &0\shrnk &0\shrnk &0\\
\nalmed
0\shrnk &0\shrnk &0\shrnk &0\shrnk &0\shrnk &0\shrnk &0\shrnk &0\shrnk &0\shrnk &0\shrnk &0\shrnk &0\shrnk &0\shrnk &0\shrnk &0\shrnk &0\shrnk &0\shrnk &1\shrnk &0\shrnk &0\shrnk &0\shrnk &0\\
\nalmed
0\shrnk &0\shrnk &0\shrnk &0\shrnk &0\shrnk &0\shrnk &0\shrnk &0\shrnk &0\shrnk &0\shrnk &1\shrnk &1\shrnk &-1\shrnk &0\shrnk &0\shrnk &0\shrnk &-1\shrnk &1\shrnk &0\shrnk &0\shrnk &0\shrnk &0\\
\nalmed
0\shrnk &0\shrnk &1\shrnk &1\shrnk &0\shrnk &1\shrnk &-2\shrnk &0\shrnk &0\shrnk &0\shrnk &0\shrnk &0\shrnk &1\shrnk &0\shrnk &1\shrnk &0\shrnk &1\shrnk &-1\shrnk &0\shrnk &-1\shrnk &-1\shrnk &0\\
\nalmed
0\shrnk &0\shrnk &0\shrnk &-1\shrnk &0\shrnk &-1\shrnk &1\shrnk &-1\shrnk &0\shrnk &0\shrnk &1\shrnk &1\shrnk &-1\shrnk &0\shrnk &0\shrnk &1\shrnk &0\shrnk &1\shrnk &1\shrnk &0\shrnk &0\shrnk &-1
\end {array} \right] 

$$}
\caption{Frobenius action on $S_X$}\label{table:Frob}
\end{table}
%
%\par
%\medskip
A smooth rational curve $Q$ on $X$ is said to be an \emph{$h_F$-conic} if $\intM{h_F, Q}{}=2$.
It is known that there exist exactly $6300$ $h_F$-conics on $X$.
See~\cite{MR3092762}.
\par
\medskip
Our next task is to calculate the period $\period_X$ of $X$ explicitly.
The discriminant group $A_S =S_X\dual/S_X$ of $S_X$ is isomorphic
to $\F_5^2$,
and is generated by
$$
\alpha_1 =[\ell_3]\dual \bmod S_X\quand 
\alpha_2 =[\ell_4]\dual \bmod S_X.
$$
With respect to the basis $\alpha_1, \alpha_2$,
 the discriminant form $q_S: A_S\to \Q/2\Z$ of $S_X$ is given by the matrix
$$
\left[ \begin {array}{cc} 2/5&0\\0&4/5
\end {array} \right]. 
$$
The automorphism group $\OG(q_S)$ of $(A_S, q_S)$ is of order $12$,
and,
by means of the basis $\alpha_1, \alpha_2$,
 each element of $\OG(q_S)$ is expressed as a right-multiplication of a $2\times 2$ matrix in $\GL_2(\F_5)$.
Consider the matrices
{\small
\begin{eqnarray*}
T_A& =&\phantom{{\vrule width 0pt height .5cm}^t }
\left[ \begin {array}{cccccccccccccccccccccc} 0&0&1&0&0&0&0&0&0&0&0&0
&0&0&0&0&0&0&0&0&0&0\\  0&0&0&1&0&0&0&0&0&0&0&0&0&0&0&0
&0&0&0&0&0&0\end {array} \right], \\
T_B& =&\phantom{{\vrule width 0pt height .7cm}}
{\vrule width 0pt height .5cm}^t 
\left[ \begin {array}{cccccccccccccccccccccc} 2&3&1&0&4&1&1&0&4&1&2&2
&4&4&0&0&1&4&1&0&0&0\\  3&2&0&1&4&2&4&3&4&1&2&4&2&1&3&0
&4&4&4&2&0&0\end {array} \right] 
\end{eqnarray*}}%
of size $2\times 22$ and $22\times 2$, respectively.
Then the action $\bar{g}\in \OG(q_S)$ on $(A_S, q_S)$ induced by an isometry $g\in \OG(S_X)$
is given by 
\begin{equation}\label{eq:barg}
\bar{g}=T_A\cdot \Gram_{S}\inv \cdot g\cdot \Gram_{S}\cdot T_B \bmod 5.
\end{equation}
%
%
%Note that $5S_X\dual \subset S_X$.
Consider the $2$-dimensional $\F_5$-vector space
$$
S_0 =5 S_X\dual/ 5 S_X \;\;\subset\;\; S_X\tensor_{\Z}\F_5.
$$
The vector space $S_0$ has a basis 
$$
\tilde{\alpha}_1 =5 [\ell_3]\dual  \bmod 5S_X
\quand
\tilde{\alpha}_2 =5 [\ell_4]\dual\bmod 5S_X,
$$
with respect to which the $\F_5$-valued quadratic form $Q_0$ is given by the matrix
$$
\left[
\begin{array}{cc}
2 & 0 \\ 0 & 4
\end{array}
\right].
$$
Recall that  
$\bar{c}_{\DR}: S_X\tensor k \to H^2_{\DR}(X)$ 
is  the Chern class map.
Then $\Ker (\bar{c}_{\DR})$ is a $1$-dimensional  isotropic subspace of $Q_0\tensor k$.
Therefore we see that   $\Ker (\bar{c}_{\DR})$ is either equal to 
$\III_+ =\gen{(1, \sqrt{2})}$ or equal to $\III_- =\gen{(1, -\sqrt{2})}$.
%Let $\varphi : S_0\tensor k \to S_0\tensor k$ denote the map $\id\tensor F_k$,
%where $F_k$ is the Frobenius of $k$. 
Since  the Frobenius map $\varphi=\id\tensor F_k$ from $S_0\tensor k$ to itself
%$$
%\varphi=\id\tensor F_k : S_0\tensor k \to S_0\tensor k
%$$ 
only interchanges $\III_+$ and $\III_-$, 
we conclude that the  period $\period_X=\varphi^*(\Ker (\bar{c}_{\DR}))$
of $X$
is either $\III_-$ or $\III_+$.
On the other hand, we have
\begin{equation*}\label{eq:same}
\set{\bar{g} \in \OG(Q_0)}{\III_+^{\bar{g}}=\III_+}=\set{\bar{g} \in \OG(Q_0)}{\III_-^{\bar{g}}=\III_-},
\end{equation*}
and this subgroup of $\OG(Q_0)$ is of order $6$
which consists of the following elements of $\GL_{2}(\F_5)$:
\begin{equation*}\label{eq:six}
\left[ \begin {array}{cc} 1&0\\  0&1\end {array}
 \right], \; 
 \left[ \begin {array}{cc} 2&1\\  3&2\end {array}
 \right], \; 
 \left[ \begin {array}{cc} 2&4\\  2&2\end {array}
 \right], \; 
 \left[ \begin {array}{cc} 3&1\\  3&3\end {array}
 \right], \;  
 \left[ \begin {array}{cc} 3&4\\  2&3\end {array}
 \right], \;  
 \left[ \begin {array}{cc} 4&0\\  0&4\end {array}
 \right].
\end{equation*}
Therefore,
for a given $g\in \OG(S_X)$,
we can determine whether
$\period_X^g=\period_X$ holds or not by 
calculating $\bar{g}$ by means of~\eqref{eq:barg}, and 
see whether $\bar{g}$ is one of the $6$ matrices above.
\par
%\medskip
For example, 
the Frobenius isometry given in Table~\ref{table:Frob}
does not preserve the period,
while the deck-transformation isometry~\eqref{eq:flip} does.
\subsection{Embedding $S_X$ into $L$}\label{subsec:embedding}
Let $\PPP_{S_X}$ be  the positive cone of $S_X$ 
containing an ample class of $X$.
We embed $S_X$ into the even unimodular hyperbolic lattice $L$ of rank $26$ primitively
in such a way that the conditions (i) and (ii) in Section~\ref{subsec:Conway}
are satisfied, and 
calculate some induced chambers contained in the $\RRR_{S_X}\sphyp$-chamber $\Nef(X)$.
%decompose the $\RRR_{S_X}\sphyp$-chamber $\Nef(X)$ into the union  of  induced chambers.
%
\begin{proposition}\label{prop:emb}
{\rm (1)} There exists a primitive embedding $S_X\inj L$
such that the orthogonal complement $R$ of $S_X$ in $L$ satisfies the conditions {\rm (i)} and {\rm (ii)}
 in Section~\ref{subsec:Conway}.
 
{\rm (2)} If $\iota: S_X\inj L$ and $\iota\sprime: S_X\inj L$
are primitive embeddings,
then there exists $g\in \OG(L)$ such that $\iota\sprime=g\circ \iota$.
\end{proposition}
\begin{proof}
By Nipp's table of reduced regular primitive positive-definite quaternary quadratic forms~\cite{MR1118842},
there exists a  negative definite lattice $R$ of rank $4$ with discriminant $25$,
and $R$ is unique up to isomorphisms.
We can choose a  basis $u_1, \dots, u_4$ of $R$
 with respect to which  the Gram matrix is
equal to
\begin{equation}\label{eq:GramR}
\left[ \begin {array}{cccc} -2&-1&0&1\\-1&-2&-1&0
\\0&-1&-4&-2\\1&0&-2&-4
\end {array} \right].
\end{equation}
It is obvious that $\RRR_R$ is non-empty.
By a direct computation,
we see that the order of $\OG(R)$ is $72$,
and obtain the list of all elements of $\OG(R)$.
\par
%A vector of $R\tensor \R$ is written 
%in terms of the \emph{dual} basis $u_1\dual, \dots, u_4\dual$ of $u_1, \dots, u_4$.
The discriminant group $A_R =R\dual/R$ of $R$ is isomorphic
to $\F_5^2$,
and is generated by
$$
\beta_1 =u_4\dual \bmod R
\quand
\beta_2 =u_2\dual \bmod R,
$$
with respect to which
the discriminant form $q_R: A_R\to \Q/2\Z$ of $R$ is given by the matrix
$$
\left[ \begin {array}{cc} 8/5&0\\0&6/5
\end {array} \right].
$$
Hence the order of $\OG(q_R)$ is $12$.
We can check by direct computation that  the natural homomorphism $\OG(R)\to \OG(q_R)$ is surjective.
\par
Recall that $\alpha_1$ and $\alpha_2$ are the basis of $A_S =S_X\dual/S_X \cong \F_5^2$ given in the previous subsection.
The linear map $\delta : A_S\to A_R$ defined by
$\delta(\alpha_1)=\beta_1$ and $\delta(\alpha_2)=\beta_2$
induces an isomorphism from $(A_S, q_S)$ to $(A_R, -q_R)$.
Consequently,
the pull-back $L$ of the graph
$$
\set{(x, \delta(x))}{x \in A_S}
$$
of $\delta$ by the natural projection $S_X\dual\oplus R\dual \to A_S\oplus A_R$
is an even unimodular hyperbolic lattice of rank  $26$,
into which $S_X$ and $R$ are primitively embedded.
(See Nikulin~\cite{MR525944}.)
%In other words,
%$L$ is the sublattice of $S_X\dual\oplus R\dual$
%generated by $S_X\oplus R$ and 
%$[\ell_3]\dual+[u_4]\dual, [\ell_4]\dual+[u_2]\dual \in S_X\dual\oplus R\dual$.
\par
The uniqueness of primitive embeddings $S_X\inj L$
up to the action of $\OG(L)$ follows from the uniqueness 
of the even negative-definite lattice of rank $4$ with discriminant $25$ and 
the surjectivity of $\OG(R)\to \OG(q_R)$. (See Nikulin~\cite{MR525944}.)
\end{proof}
In the following,
we use the primitive embedding $S_X\inj L$
constructed in the proof of Proposition~\ref{prop:emb}.
Let $\PPP_L$ be the positive cone containing $\PPP_{S_X}$.
An element  of $L\tensor \R$ is written in the form of a vector
$[x_1, \dots, x_{26}]\dual$
with respect to the basis $[\ell_1]\dual, \dots, [\ell_{22}]\dual$, $[u_1]\dual, \dots, [u_4]\dual$
of $S_X\dual \oplus R\dual$.
\par
\medskip
Let $w$ be a Weyl vector of $L$ such  that the corresponding Conway chamber $\DDD(w)$ is $S_X$-nondegenerate,
and let $D$ denote the chamber $\DDD(w)\cap \PPP_{S_X}$ of $\PPP_{S_X}$ induced by $\DDD(w)$.
We denote by $\Walls(D)$ the set of walls of $D$.
For a wall $W\in \Walls(D)$,
there exists a unique \emph{primitive} vector $v_W\in S_X\dual$
such that $W=(v_W)\sperp$ and $\intM{v_W,u}{}>0$,
where $u$ is a point in the interior of $D$.
A wall $W\in \Walls(D)$ is said to be \emph{of type $[a,n]$}
if  $\intM{v_W,w_S}{}=a$ and  $\intM{v_W,v_W}{}=n$,
where $w_S\in S_X\dual$ is the projection of the Weyl vector $w\in L$.
Suppose that $D$ is contained in the $\RRR_{S_X}\sphyp$-chamber $\Nef(X)$.
Then a wall  $W\in \Walls(D)$ of type $[a, n]$ is a wall  of $\Nef(X)$ if and only if
there exists an integer $c$ such that $ac=1$, $nc^2=-2$ and $c\, v_W\in S_X$.
\par
Let $D$ be an  induced chamber contained in $\Nef(X)$,
and
let  $h_D\in S_X$ be a vector contained in the interior of $D$ that is invariant under the action of 
$\aut(D)$.
Then $h_D$ is ample, and 
$$
\aut_{X}(D)=\aut(D)\cap G_X=\set{g\in \OG(S_X)}{D^g=D,\;\; \period_X^g=\period_X}
$$
is the automorphism group of the polarized $K3$ surface $(X, h_D)$.
%corresponding to the induced chamber $D$.
%
\subsection{The induced chamber $D_0$}\label{subsec:D0}
We put
\begin{equation}\label{eq:Weyl-fermat}
w_0 =h_F+u_1 \in S_X\oplus R\subset L.
\end{equation}
Since $w_0$ is primitive in $L$,
$w_0$ belongs to $\closure{\PPP}_L$, and  $\gen{w_0}\sperp/\gen{w_0}$  contains no $(-2)$-vectors,
we see that $w_0$ is a Weyl vector.
We denote by $\pr_{S_X}$ the orthogonal-projection from $L\otimes \R$ to $S_{X}\otimes \R$.
Calculating the finite set 
$$
\pr_{S_X}(\Delta(w_0))\cap \RRR_{L|S} \;\;=\;\; \shortset{r_{S_X}}{r\in \Delta(w_0),\;\; \intM{r_{S_X}, r_{S_X}}{S_X}<0},
$$
we see that  $h_F=w_{0, S}$ belongs  in the interior of 
$$
D_0 =\DDD(w_0)\cap \PPP_{S_X}.
$$
Hence the Conway chamber $\DDD(w_0)$ is $S_X$-nondegenerate and $D_0$ is an induced chamber.
The order of $\aut_{X}(D_0)$ is $756000$,
and it coincides with the  automorphism group  of
the Fermat double sextic plane $(X, h_F)$.
%The order of $\aut(D_0)$ is $2\cdot 756000$,
%and it is generated by 
% $\aut_{X}(D_0)$ and the Frobenius isometry given in Table~\ref{table:Frob}.
The action of $\aut_{X}(D_0)=\Aut(X, h_F)$ 
decomposes the set $\Walls(D_0)$ of walls of $D_0$ into the union of three orbits $O_{0,0}$, $O_{0,1}$, $O_{0,2}$
described as follows:
$$
\begin{array}{c|cc}
\textrm{no.} & \textrm{type} & \textrm{card.} \\ 
\hline0& [1, -2] & 252\\ 
1& [1, -8/5] & 300\\ 
2& [2, -6/5] & 15750\\ 
\end{array}

$$
The walls in the orbit $O_{0,0}$ of cardinality $252$ are walls of $\Nef(X)$,
and hence they correspond to  
smooth rational curves on $X$.
Let $R_{252}$ denote the set of smooth rational curves on $X$
corresponding to the walls in $O_{0, 0}$.
Then $R_{252}$ coincides with the set of $h_F$-lines.
\subsection{The induced chamber $D_1$}\label{D_1}
The $\aut_X(D_0)$-orbit $O_{0,1}$  of the walls of $D_0$ contains a wall $(v_1)\sperp$, where 
$$
v_1 =[0, 1, 1, 0, 0, 1, 0, 1, 0, 1, 1, 0, 1, 0, 0, 1, 1, 1, 1, 1, 1, 1]\dual\in S_X\dual.
$$
We put
$$
w_1 =[1, 2, 2, 1, 1, 2, 1, 2, 1, 2, 2, 1, 2, 1, 1, 2, 2, 2, 2, 2, 2, 2, 2, 1, 1, 0]\dual \in L.
$$
Then $w_1$ is a Weyl vector, the Conway chamber $\DDD(w_1)$ is $S_X$-nondegenerate,
and the induced chamber 
$$
D_1 =\DDD(w_1)\cap \PPP_{S_X}
$$
is adjacent to $D_0$
along the wall $(v_1)\sperp$.
The vector 
$w_{1, S} \in S_X\dual$
is contained in the interior of $D_1$ and satisfies
$w_{1, S}^2=12/5$.
We put $h_1 =5 w_{1, S}$.
Then 
$$
h_1=[14, 16, -4, -6, -5, -11, 12, -8, -5, 0, 10, 8, -13, 3, -3, 5, -8, 10, 7, -2, 5, -10] %\in S_X
$$
is a polarization of degree $60$.
The degree $\intM{h_F, h_1}{}$ of the polarization $h_1$ with respect to
$h_F$ is $15$.
The  automorphism group $\aut_{X}(D_1)$ of the polarized $K3$ surface $(X, h_1)$ is 
of oder $20160$.
The action of $\aut_{X}(D_1)$ 
decomposes $\Walls(D_1)$ into the union of $18$ orbits $O_{1, 0}, \dots, O_{1, 17}$
described as follows:
$$
\begin{array}{c|cc}
\textrm{no.} & \textrm{type} & \textrm{card.} \\ 
\hline0& [1, -2] & 168\\ 
1& [3/5, -8/5] & 8\\ 
2& [4/5, -8/5] & 15\\ 
3& [4/5, -8/5] & 15\\ 
4& [6/5, -8/5] & 70\\ 
5& [6/5, -8/5] & 70\\ 
6& [7/5, -8/5] & 168\\ 
7& [9/5, -6/5] & 280\\ 
8& [9/5, -6/5] & 280\\ 
\end{array}
\qquad
\begin{array}{c|cc}
\textrm{no.} & \textrm{type} & \textrm{card.} \\ 
\hline9& [2, -6/5] & 840\\ 
10& [2, -6/5] & 840\\ 
11& [11/5, -6/5] & 1680\\ 
12& [11/5, -6/5] & 1680\\ 
13& [11/5, -6/5] & 840\\ 
14& [11/5, -6/5] & 840\\ 
15& [8/5, -4/5] & 15\\ 
16& [8/5, -4/5] & 15\\ 
17& [9/5, -2/5] & 8\\ 
\end{array}

$$
We confirm by computer that 
the action of $\aut_{X}(D_1)$ on the orbit $O_{1,1}$ of cardinality $8$ embeds 
$\aut_{X}(D_1)$ into the symmetric group $\mathfrak{S}_8$.
Hence $\aut_{X}(D_1)$ is isomorphic to 
the alternating group $\mathfrak{A}_8$.
\par
The  wall $(v_1)\sperp$ separating 
$D_0$ and $D_1$ is a member of the orbit $O_{1,1}$.
Hence $D_1$ is adjacent to eight induced chambers $G_X$-congruent to $D_0$.
Moreover we have
$$
|\Aut_X(D_0)\cap \Aut_X(D_1)|=\frac{|\Aut_X(D_0)|}{300}=\frac{|\Aut_X(D_1)|}{8}=2520.
$$
\par
The walls in the orbit $O_{1,0}$ are walls of $\Nef(X)$,
and hence they correspond to  
smooth rational curves on $X$.
Let $R_{168}$ denote the set of smooth rational curves on $X$
corresponding to the walls in $O_{1,0}$.
We observe the following facts by a direct calculation:
\begin{proposition}\label{Shimada168}
Any  distinct two curves in $R_{168}$ are either disjoint  or intersecting at one point transversely.
For any curve $\varGamma$ in $R_{168}$,
there exist exactly $72$ curves in $R_{168}$ that intersect $\varGamma$.
\end{proposition}
\begin{proposition}\label{Shimada42}
Among $R_{168}$,
exactly $126$ curves are contained in the set $R_{252}$
of $h_F$-lines,
while the other $42$ curves are $h_F$-conics.
The deck-transformation of $X_F\to\P^2$ maps 
$R_{252}\cap R_{168}$ to 
the complement $R_{252}\setminus (R_{252}\cap R_{168})$
bijectively.
\end{proposition}
\subsection{The induced chamber $D_2$}
The $\aut_X(D_0)$-orbit $O_{0,2}$  of the walls of $D_0$ 
 contains a wall $(v_2)\sperp$, where 
$$
v_2=[1, 1, 2, 1, 0, 1, 1, 1, 1, 1, 2, 0, 1, 1, 1, 2, 2, 1, 1, 1, 2, 2]\dual\in S_X\dual.
$$
We put
$$
w_2 =[4, 4, 7, 4, 1, 4, 4, 4, 4, 4, 7, 1, 4, 4, 4, 7, 7, 4, 4, 4, 7, 7, 2, 1, -1, 0]\dual \in L.
$$
Then $w_2$ is a Weyl vector, the Conway chamber $\DDD(w_2)$ is $S_X$-nondegenerate,
and the induced chamber 
$$
D_2 =\DDD(w_2)\cap \PPP_{S_X}
$$
is adjacent to $D_0$
along the wall $(v_2)\sperp$. 
The vector $w_{2, S} \in S_X\dual$ 
is contained in the interior of $D_2$ and satisfies
$w_{2, S}^2=16/5$.
We put $h_2 =5 w_{2, S}$. Then 
$$
h_2=[14, 11, 3, 6, 21, 15, -3, 18, 6, -6, -27, 0, 9, -12, 3, -15, -3, -9, -18, 12, 0, 15]
$$
is a polarization of degree $80$.
The degree $\intM{h_F, h_2}{}$ of the polarization $h_2$ with respect to
$h_F$ is $40$.
The  automorphism group $\aut_{X}(D_2)$ of the polarized $K3$ surface $(X, h_2)$ is of order $1152$.
The action of $\aut_{X}(D_2)$ 
decomposes $\Walls(D_2)$ into the union of twenty seven orbits $O_{2,0}, \dots, O_{2,26}$ described as follows:
$$
\begin{array}{c|cc}
\textrm{no.} & \textrm{type} & \textrm{card.} \\ 
\hline0& [1, -2] & 48\\ 
1& [1, -2] & 48\\ 
2& [2/5, -8/5] & 4\\ 
3& [2/5, -8/5] & 4\\ 
4& [1, -8/5] & 16\\ 
5& [1, -8/5] & 16\\ 
6& [8/5, -8/5] & 72\\ 
7& [8/5, -8/5] & 72\\ 
8& [8/5, -8/5] & 64\\ 
\end{array}
\qquad
\begin{array}{c|cc}
\textrm{no.} & \textrm{type} & \textrm{card.} \\ 
\hline9& [8/5, -8/5] & 64\\ 
10& [8/5, -6/5] & 24\\ 
11& [9/5, -6/5] & 48\\ 
12& [9/5, -6/5] & 48\\ 
13& [9/5, -6/5] & 16\\ 
14& [9/5, -6/5] & 16\\ 
15& [11/5, -6/5] & 288\\ 
16& [11/5, -6/5] & 288\\ 
17& [11/5, -6/5] & 96\\ 
\end{array}
\qquad
\begin{array}{c|cc}
\textrm{no.} & \textrm{type} & \textrm{card.} \\ 
\hline18& [11/5, -6/5] & 96\\ 
19& [11/5, -6/5] & 48\\ 
20& [11/5, -6/5] & 48\\ 
21& [12/5, -6/5] & 576\\ 
22& [12/5, -6/5] & 192\\ 
23& [12/5, -6/5] & 192\\ 
24& [12/5, -6/5] & 144\\ 
25& [8/5, -4/5] & 3\\ 
26& [8/5, -4/5] & 3\\ 
\end{array}

$$
\par
The wall $(v_2)\sperp$ separating 
$D_0$ and $D_2$ is a member of the orbit $O_{2,10}$.
Hence $D_2$ is adjacent to $24$ induced chambers $G_X$-congruent to $D_0$.
Moreover we have
$$
|\Aut_X(D_0)\cap \Aut_X(D_2)|=\frac{|\Aut_X(D_0)|}{15700}=\frac{|\Aut_X(D_2)|}{24}=48.
$$
\par
The walls in the orbits $O_{2,0}$ and  $O_{2,1}$ are walls of $\Nef(X)$,
and hence they correspond to 
smooth rational curves on $X$.
Let $R_{48,0}$ and $R_{48,1}$ denote the sets of smooth rational curves on $X$
corresponding to the walls in $O_{2,0}$ and  $O_{2,1}$, respectively.
We observe the following facts:
\begin{proposition}
Any distinct two curves in the union $R_{48, 0}\cup R_{48, 1}$ are
either disjoint  or intersecting at one point transversely.
For $\nu=0,1$, the set $R_{48, \nu}$ is a union  of  
 three sets $\SSS_{\nu 0}, \SSS_{\nu 1}, \SSS_{\nu 2}$ of disjoint $16$ smooth rational curves.
Each $\SSS_{\nu j}$ contains eight $h_F$-lines,
and the $h_F$-degree of the remaining eight  smooth rational curves is $4$.
We can number these  six sets so that they satisfy the conditions $(a), (b), (c)$ in Theorem~\ref{thm:six}.
\end{proposition}
We remark the following fact.
\begin{proposition}
Let $\SSS$ and $\SSS\sprime$ be sets of disjoint $16$ smooth rational curves on $X$.
Then there exists $g\in \Aut(X)$ such that $g(\SSS)=\SSS\sprime$.
\end{proposition}
\begin{proof}
By Nikulin~\cite{MR0429917},
if $\SSS_Y$ is a set of disjoint $16$ smooth rational curves
on a $K3$ surface $Y$ in characteristic $\ne 2$,
then $Y$ is a Kummer surface associated with
an abelian surface $A$ 
and $\SSS_Y$ is the set of exceptional curves
of the minimal resolution $Y\to A/\gen{\iota_{A}}$.
(The proof in Nikulin~\cite{MR0429917} is valid not only over $\C$ but also in odd characteristics.)
\par
Let $\zeta: X\to Z$ and $\zeta\sprime: X\to Z\sprime$
be the contractions of the $(-2)$-curves in $\SSS$ and $\SSS\sprime$,
respectively.
Then
there exist abelian surfaces $A$ and $A\sprime$
such that $Z\cong A/\gen{\iota_A}$ and $Z\sprime\cong A\sprime/\gen{\iota_{A\sprime}}$,
where $\iota_A$ and $\iota_{A\sprime}$ are the inversions of $A$ and $A\sprime$,
respectively.
By~\cite{MR555718},
both of $A$ and $A\sprime$ are superspecial.
Since a superspecial abelian surface is unique up to isomorphisms in characteristic $5$  by~\cite{MR555718},
there exists an isomorphism $f : A\isom A\sprime$ of abelian surfaces.
Since $f\circ \iota_A=\iota_{A\sprime}\circ f$,
the isomorphism $f$ induces  $A/\gen{\iota_A} \isom A\sprime/\gen{\iota_{A\sprime}}$,
and therefore we obtain an isomorphism $g\sprime : Z\isom Z\sprime$.
Since $X$, $Z$ and $Z\sprime$ are birational and $X$ is minimal,
there exists $g\in \Aut(X)$
such that $\zeta\sprime\circ g=g\sprime \circ \zeta$ holds.
We obviously have $g(\SSS)=\SSS\sprime$.
\end{proof}
%
%Therefore, for any $i, j=0,1,2$,
%there  exists $g\in \Aut(X)$ such that $g(\SSS_{0i})=\SSS_{1j}$.
%
\subsection{Further induced chambers}\label{subsec:further}
We define the \emph{level} of an induced chamber $D$ to be the minimal non-negative integer $\ell$ such that
there exists a chain
$$
D^{(0)}=D_0, D^{(1)}, \dots, D^{(\ell)}=D
$$ 
from  $D_0$ to $D$
of induced chambers such that $D^{(i-1)}$ and $D^{(i)}$ are adjacent.
The \emph{level} of a $G_X$-congruence class  of induced chambers 
is defined to be  the minimum of the levels of elements of the class.
We have made the list of the $G_X$-congruence classes of  induced chambers  of level $< 4$.
The number is
$$
\begin{array}{cc}
\text{level} & \text{number of $G_X$-conguence classes}\\
0 & 1 \\
1 & 2\\
2 & 12\\
3 & 328
\end{array}
$$
For level $4$, we found more than six thousand  $G_X$-congruence classes,
and hence we have given up the computation.
The data of the  induced chambers $D_i$ of level $2$ are presented in Table~\ref{table:level2}.
The third column is the orbit decomposition of the $(-2)$-walls of $D_i$ by the action of $\Aut_X(D_i)$.
\begin{table}
$$
\begin{array}{ccc}
i & |\Aut_X(D_i)| & \textrm{orbits of $(-2)$-walls} \\
\hline
3 &360& [18, 60]\\
4 &36& [6, 9, 18, 18]\\
5 &36& [6, 9, 18, 18]\\
6 &48& [6, 8, 12, 24]\\
7 &48& [6, 8, 12, 24]\\
8 &72& [3, 12, 12, 18]\\
9 &12& [3, 6, 6, 6, 6, 12]\\
10 &8& [2, 2, 2, 4, 4, 4, 4, 4, 8, 8]\\
11 &2& [1, 1, 2, 2, 2, 2, 2, 2, 2, 2, 2, 2, 2, 2, 2, 2, 2, 2]\\
12 &6& [2, 2, 3, 3, 3, 3, 6, 6, 6]\\
13 &6& [2, 2, 3, 3, 3, 3, 6, 6, 6]\\
14 &8& [2, 4, 4, 4, 4, 8, 8]
\end{array}
$$
\caption{Induced chambers of level $2$}\label{table:level2}
\end{table}
In level $3$, we have found many induced chambers $D_i$ with $ |\Aut_X(D_i)|=1$.
\begin{remark}
In~\cite{ShimadaChar5},
various sextic double plane   models of $X$ are systematically investigated by another method.
\end{remark}

%%%%%%%%%%%

%
%
\section{Proof of Theorems by lattice theory}\label{sec:bylattice}
In this section, we  prove Theorems~\ref{thm:main}~and~\ref{thm:six} by using 
lattice theory.  
To do this, we give three primitive embeddings of $S_X$ into 
the even unimodular lattice $L$ of signature $(1,25)$
 corresponding to the three cases in
Theorem~\ref{thm:main}, and then apply the Borcherds method and a theory of the Leech lattice. 
\par
\medskip
First of all we fix the notation. 
We denote by $\Lambda$ the unique even \emph{negative}-definite unimodular lattice of rank $24$ without $(-2)$-vectors;
that is,  $\Lambda$ is the \emph{Leech lattice}.  
In the following, we recall an explicit description of $\Lambda$
briefly.
%Recall that $\Lambda$ is realized as a certain subgroup  in  ${\R}^{24} =
%{\R}^{{\P}^1({\F}_{23})}$ equipped with inner product $\langle x,y \rangle= -\frac{x\cdot y}{8}$.  
Let $\Omega= \{ \infty, 0, 1,..., 22 \}$ be the projective line ${\P}^1({\F}_{23})$ 
over the field ${\F}_{23}$.  
We consider the set $P(\Omega)$ of
all subsets of $\Omega$ with the symmetric difference as a $24$-dimensional vector space over ${\F}_{2}$. 
Let ${\CCC}$ be the \emph{binary Golay code},  
which is a $12$-dimensional subspace of 
$P(\Omega)$.
% defined by the 24 sets $N_{i}$ where $N = \Omega \setminus 
%\{ x^{2} : x \in {\F}_{23} \}$, $N_{\infty} = \Omega$ and $N_{i} = 
%\{ n - i : n \in N \}$ for $i \ne  \infty$  (see Conway~\cite{MR0338152}).
We call a set in ${\CCC}$ a \emph{${\CCC}$-set}.  
A ${\CCC}$-set 
consists of $0$, $8$, $12$, $16$ or $24$ elements.  
An $8$-elements ${\CCC}$-set 
%(resp. 12-elements ${\CCC}$-set)
is called an \emph{octad},
% (resp. \emph{dodecad}), 
and a set of $6$ tetrads is 
called a \emph{sextet} if the union of any two tetrads is 
an octad.  We denote by ${\CCC}(8)$ 
%(resp. ${\CCC}(12)$) 
the set of all octads.
Let ${\R}^{24}$ be spanned by an
ortho-normal  basis $\nu_{i}$ ($i \in \Omega$).  For  a subset $S \subset \Omega$, we define $\nu_{S}$ to be
$\sum_{i\in S} \nu_i$ .  
Then the Leech lattice
$\Lambda$
is the lattice generated by the vectors $2\nu_{K}$ for $K \in {\CCC}(8)$ and
$\nu_{\Omega} - 4\nu_{\infty}$
with the symmetric bilinear form
$$
\gen{x, y} = - \frac{x \cdot y}{8}.
$$  
%Here, for $x, y \in \Lambda$, we define $\langle x, y \rangle = {-x \cdot y\over 8}$.  
%
\begin{proposition}[Conway~\cite{MR0338152}, Section 4, Theorem 2]\label{prop:leech1} %{\rm (Conway~\cite{MR0338152}, \S 4, Theorem 2)}
A vector 
$(\xi_{\infty}, \xi_{0},..., \xi_{22})$ with $\xi_i\in \Z$ is in $\Lambda$ if and only if

{\rm (i)}   the coordinates $\xi_{i}$ are all congruent modulo $2$, to  $m$, say;

{\rm (ii)} the set of $i$ for which $\xi_{i}$ takes any given value modulo 
$4$ is a ${\CCC}$-set;

{\rm (iii)} the coordinate-sum is congruent to $4m$ modulo $8$. 
\end{proposition}

We denote by $\Lambda_{n}$
the set of all vectors $x$ in $\Lambda$ with $\langle x, x \rangle = -n$.  
Note  that $\Lambda_{2}=\emptyset$.
%We remark that $\Lambda_{2}$ is empty.
%
\begin{proposition}[Conway-Sloane~\cite{MR1662447}, p.133, Table 4.13]
\label{prop:leech2} %{\rm (Conway-Sloane~\cite{MR1662447}, p.133, Table 4.13)}
The  complete lists of $\Lambda_{4}$, $\Lambda_{6}$ are as follows:
\begin{eqnarray*}
\Lambda_{4} &=& \{ (\pm 2^{8}, 0^{16}), (\pm 3, \pm 1^{23}), (\pm 4^{2}, 0^{22})\},\\
\Lambda_{6} &=& \{ (\pm 2^{12}, 0^{12}), (\pm 3^{3}, \pm 1^{21}),
(\pm 4, \pm 2^{8}, 0^{15}), (\pm 5, \pm 1^{23}) \},
\end{eqnarray*}
where the signs are taken to satisfy the conditions in Proposition~\ref{prop:leech1}.
\end{proposition}
We fix a decomposition 
\begin{equation}\label{decomp}
L = U \oplus \Lambda,
\end{equation}
where $U$ is the even unimodular hyperbolic lattice of rank $2$
with the Gram matrix
$$
\left[\begin{array}{cc} 0 & 1 \\ 1 & 0 \end{array}\right].
$$
We write $(m,n, \lambda)$ for a vector in $L$,
 where $\lambda$ is in $\Lambda$, and $m, n$ are integers.
Then 
its norm is given by $2mn + \langle \lambda , \lambda \rangle$.  
We take a vector $w=(1,0,0)$ as a Weyl vector.
 Then a $(-2)$-vector $r$ in $L$ with $\langle r, w \rangle = 1$ is
called a \emph{Leech root}. 
Let $\DDD$ be the Conway chamber with respect to $w$.   Then
the automorphism group of $\DDD$
$$
\Aut(\DDD) = \set{ g \in \OG(L)}{\DDD^g =\DDD}
$$
is isomorphic to the affine automorphism group of $\Lambda$:
$$\Aut(\DDD) \cong \Lambda\rtimes \OG(\Lambda).$$
The set of all Leech roots bijectively
corresponds to the set $\Lambda$ as follows 
(Conway-Sloane~\cite{MR1662447}, Chapter~26, Theorem~3):
$$
L \ni r = (-1 - \langle \lambda, \lambda \rangle/2, 1, \lambda)  \;\;\longleftrightarrow \;\; \lambda \in \Lambda.
$$
\begin{remark}%\label{weyl-4}
{For Leech roots $r, r' \in L$ and the corresponding 
vectors $\lambda, \lambda^{'}$ in $\Lambda$, 
$\langle r, r' \rangle = 0$ if $\lambda - \lambda^{'} \in \Lambda_{4}$ 
and $\langle r, r' \rangle = 1$ if $\lambda - \lambda^{'} \in \Lambda_{6}$.
}
\end{remark}
\subsection{Proof of Theorem~\ref{thm:main}\,(1)}
We consider the following vectors in the Leech lattice $\Lambda$: 
\begin{equation}\label{XYZ1}
A = 4\nu_{\infty} +\nu_{\Omega},\;\; 
B=0, \;\; 
C=2\nu_{K_0}, \;\; 
 D = 4\nu_0 + \nu_{\Omega}, 
\end{equation}
where $K_0$ is an octad with $\infty\notin K_0$ and $0 \in K_0$.
Note that
$$A^2=D^2 =-6, \ \ C^2=-4,\ \ \langle A, C\rangle = -2, \ \ \langle A, D\rangle = -4, \ \ \langle C, D\rangle = -3.$$
Consider  the vectors in $L = U \oplus \Lambda$ defined by
\begin{equation}\label{xyz1}
a = -(2, 1, A),
\;\; 
 b = (-1, 1, 0), 
 \;\; 
c = (0, 1, C), 
 \;\; 
  d = (1,1, D).
\end{equation}
Obviously we have 
\begin{eqnarray*}
&&a^2= b^2 = -2, \;\;
 c^2 = d^2 =-4,  \;\;
  \langle a, b \rangle = \langle b, c \rangle = -1,  \;\; \\
&&\langle a, d \rangle = 1, \;\; \langle c, d \rangle =-2,  \;\;
\langle a, c \rangle = \langle b, d \rangle = 0.
\end{eqnarray*}
Let $R_1$ be the sublattice of $L$ generated by $a, b, c, d$.  
Note that the Gram matrix of $R_1$ is the same as the one given in~\eqref{eq:GramR}.
Obviously $R_1$ is primitive in $L$.   
Let $S_1$ be the orthogonal complement of $R_1$ in $L$. 
Then the signature of $S_1$ is $(1, 21)$ and $S_1\dual/S_1 \cong R_1\dual/R_1 \cong ({\Z}/5{\Z})^2$.
Thus $S_1$ is isomorphic to the N\'eron-Severi lattice $S_X$ of the supersingular $K3$ surface $X$ with Artin invariant $1$ in characteristic $5$.
\begin{lemma}\label{weyl3}
Let $w'$ be the projection of the Weyl vector $w$ into $S_1\dual$.  Then $w' \in S_1$ and $(w')^2=2$.  
Moreover $w'$ is conjugate to the class of an ample divisor under the action of $\Wgr (S_1)$.
\end{lemma}
\begin{proof}
Denote by $w''$ the projection of $w$ into $R_1\dual$.  By definition (\ref{xyz1}), we have  $\langle w'', a \rangle = -1$
and $\langle w'', b\rangle = \langle w'', c \rangle = \langle w'', d \rangle = 1$.  This implies that 
$w'' = a-b \in R_1$.  Hence $w' = w - w'' \in S_1$ and $(w')^2 = 2$.
Let $r$ be any $(-2)$-vector in $S_1$.  Then, under the embedding $S_1 \subset L$,
$r$ is a $(-2)$-vector in $L$.  Therefore 
$\langle r, w' \rangle = \langle r, w\rangle \ne 0$.
Hence we have the last assertion.
\end{proof}
Now we  determine all smooth rational curves on $X$ whose degree with respect to $w'$ is minimal.
Note that such curves correspond to all Leech roots perpendicular to $R_1$ under the above embedding $S_1 \subset L$.
\begin{lemma}\label{252}
There exist exactly $252$ Leech roots which are orthogonal to $R_1$.
\end{lemma}
\begin{proof}
Let $r$ be a Leech root perpendicular to $R_1$.  The condition $\langle r, b\rangle = 0$
implies $r = (1,1, \lambda)$ with $\lambda \in \Lambda_4$.  Similarly we have
\begin{equation}\label{cond}
\langle \lambda, A\rangle = -3, \;\;
\langle \lambda, C \rangle = -1, \;\;
\langle \lambda, D \rangle = -2.
\end{equation}
Now we use Proposition~\ref{prop:leech1}.  If $\lambda = \pm 4\nu_i \pm 4\nu_j$, then
the condition $\langle \lambda, A \rangle = -3$ implies that $\lambda = 4\nu_{\infty}+ 4\nu_i$.  
Then $\langle \lambda, D\rangle = -1$ or $-3$.  This contradicts~\eqref{cond}.
\par
If $\lambda = (\pm 2^8, 0^{16})$, then the condition $\langle \lambda, A\rangle = -3$
implies that $\lambda = 2\nu_K$ where $K$ is an octad containing $\infty$.
The condition $\langle \lambda, D \rangle = -2$ implies that $K$ does not contain $0$, and finally the condition $\langle \lambda, C \rangle = -1$ implies that $|K_0\cap K| =2$.
\par
If $\lambda = (\pm 3, \pm 1^{23})$, we first show that the case $\lambda = (-3, \pm 1^{23})$ does not occur.  Assume $\lambda = (-3, \pm 1^{23})$.  Since 
$\langle \lambda, A\rangle = -3$, we have $\lambda = (-3, 1^{23}) = \nu_{\Omega} - 4\nu_i$,
$i\ne \infty$.  Then $\langle \lambda, D \rangle = -1$ or $-3$.  
This contradicts the condition~\eqref{cond}.  
Now assume that $\lambda = (3, \pm 1^{23})$.
Since $\langle \lambda, A\rangle = -3$, we have  $\lambda = 4\nu_{\infty} + \nu_{\Omega} - 2\nu_K$ where $K$ is an octad containing $\infty$.  
The condition 
$\langle \lambda, D\rangle = -2$ implies that $K$ does not contain $0$.  Finally the condition $\langle \lambda, C \rangle = -1$ implies that $|K\cap K_0| = 2$.
\par
Thus the desired Leech roots are 
$$
(1,1,2\nu_K) \;\; {\rm and} \;\; (1,1, 4\nu_{\infty} + \nu_{\Omega} -2\nu_K) = (1,1, A -2\nu_K)
$$
where $K$ is an octad such that $\infty \in K$, $0\notin K$ and $|K \cap K_0|=2$. 
\par
In the following,  we show that there exist exactly $126$ such octads $K$.
Let $a_1, a_2$ be in $K_0\setminus \{0\}$.  Then the number of octads containing three points $\infty, a_1, a_2$ is 21 (see Conway~\cite{MR0338152}, Theorem 11).
Take two points $a_3, a_4 \in K_0\setminus \{a_1, a_2\}$.  Then there exists exactly one octad containing $5$ points $\infty, a_1, a_2, a_3, a_4$.
Thus the number of octads $K$ containing $\infty, a_1, a_2$ and satisfying $K\cap K_0 =\{a_1, a_2\}$ is $21 - {6\choose 2}  = 6$.  Therefore
the number of octads $K$ containing $\infty$ and satisfying $|K\cap K_0| = 2$ is ${7\choose 2}\times 6 = 126$.
\end{proof}
\begin{theorem}
For a suitable identification of $S_1$ with $S_X$, $(X, w')$ is isomorphic to
$(X, h_F)$.
\end{theorem}
\begin{proof}
Recall that 
we have given a primitive embedding of $S_1$ into $L$ 
with a Weyl vector $w$ whose orthogonal complement is $R_1$ (see (\ref{xyz1})). 
On the other hand, we have given a primitive embedding of $S_X$ into $L$ 
with a Weyl vector $w_0$ whose orthogonal
complement is $R$ (see~\eqref{eq:Weyl-fermat}).  
We identify these two embeddings as follows.
First we use the decomposition $L=U\oplus \Lambda$ given in (\ref{decomp}) and
we may assume that $R$ is generated by
$$
u_1=a-b, \;\;  u_2= -b, \;\; u_3=-c+d, \;\; u_4 =d, 
$$
where $\{u_1,u_2,u_3,u_4\}$ is a basis of $R$ with the Gram matrix~\eqref{eq:GramR}.
Obviously $R=R_1$.
Then $S_X = R^{\perp}$.
The Weyl vector $w_0 = h_F+u_1$ and $u_2$ generate a hyperbolic plane $U' (\cong U)$ in $L$,
and hence we have a decomposition
$$
L=U' \oplus \Lambda', 
$$
where $\Lambda' = U'^{\perp} \cong \Lambda$.  Write $w_0=(1,0,0)$ and $u_2 = (1,-1,0)$ with respect to the
decomposition $L=U'\oplus \Lambda'$.  
Since $\langle w_0, a \rangle = -1$
and $\langle u_2, a \rangle = 1$, we have
$$
a =(-2,-1, -A'),
$$
where $A'\in \Lambda'$ satisfies $A'^2 = -6$.
Similarly we have
\begin{eqnarray*}
&& b=(-1,1,0),\\
&& c =(0,1,C'), \;\;\text{\rm where}\;\;  C'\in \Lambda', \;\; C'^2=-4,\\
&& d =(1,1,D'), \;\;\text{\rm where}\;\; D'\in \Lambda', \;\; D'^2= -6,\\
&& \langle A', C'\rangle = -2, \ \ \langle A', D'\rangle = -4, \ \ \langle C', D'\rangle = -3.
\end{eqnarray*}
Note that $A', B' (=0), C', T'$ define a root lattice $A_4$ in $\Lambda'$ in the sense of the paper~\cite{MR913200}; that is, the following Leech roots with respect to $w_0$
$$
(2,1,A'), \;\; (-1,1,0), \;\; (1,1,C'),  \;\; (2,1,D')
$$
generate a root lattice in $U'\oplus \Lambda'$.  It follows from Lemma 6.1 in~\cite{MR913200} that $\Aut(\DDD)$ acts transitively on the set of root lattices of type $A_4$, where $\DDD$ is the Conway chamber with respect to  the Weyl vector
$w_0=(1,0,0) \in U'\oplus \Lambda'$.  Since $\Aut(\DDD)$ fixes $w_0$, we may
assume that $A',B',C',D'$ coincide with $A,B,C,D$ given in (\ref{XYZ1}).
Thus we have shown that the embedding of $S_X$ into $L$ is the same one given in (\ref{xyz1}) and hence $h_F = w'$.
\end{proof}
%Finally the 252 Leech roots are the classes of the 252 smooth rational curves on $X$ because Leech roots have the minimal degree 1 with respect to the Weyl vector $w$.
%Thus we have finished the proof of Theorem 1.1\,(1).

\begin{remark}\label{weyl-3}
Let $r = (1,1, 2\nu_K)$ and $r'= (1,1, A-2\nu_K)$ be Leech roots as in the proof of Lemma~\ref{252}.  
In the proof of Lemma~\ref{weyl3}, we showed that $w''= a-b$.
Hence we have 
$$w' = w - w'' = (1,0,0) + (2,1,A) +(-1,1,0) = (2,2,A) = r + r'.$$
Thus we have $w' = r + r'$ and $\langle r, r'\rangle = 3$.
%where $w'$ is the projection of the Weyl vector $w$ into $S_1$.  
This corresponds to the fact that
the pullback of the tangent line of the Fermat sextic curve $C_F$ at an $\F_{25}$-rational point 
under the degree two map 
$\pi_F: X \to \P^2$ splits into two smooth rational curves
meeting at one point with multiplicity $3$.
\end{remark}

We know that the projective automorphism group $\Aut(X, w')$ is
a central extension of $\PGU(3, \F_{25})$ by the cyclic group of order $2$ generated by the deck-transformation of $X$ over $\P^2$.  Here we show that the subgroup 
${\PSU}(3,\F_{25})$ of index $6$ acts on $X$ by using the Torelli theorem for supersingular $K3$ surfaces.

\begin{proposition}\label{group3} 
The group ${\PSU}(3,\F_{25})$ acts on $X$ %as automorphisms.
by automorphisms.
\end{proposition}
\begin{proof}
First we see that the point-wise stabilizer of $\{A, B, C, D\}$ of ${\OG}(\Lambda)$ is 
${\PSU}(3,\F_{25})$.
The point-wise stabilizer of the three points $\{A=4\nu_{\infty} + \nu_{\Omega}, B=0, D=4\nu_0 + \nu _{\Omega}\}$
is the Higman-Sims group 
$\mathord{\rm HS}$ (see Conway~\cite{MR0338152}, 3.5).  
It is known that there exist 352 vectors $C'$ in $\Lambda$ satisfying
$$A-C' \in \Lambda_6 \;\; {\rm  and} \;\; B-C', D-C' \in \Lambda_4.$$
Note that $C=2\nu_{K_0}$ is one of them.
Moreover they form 176 pairs $\{ C', D - C'\}$ (Conway~\cite{MR0338152}, 3.5). 
It follows from the table of maximal subgroups in Atlas~(page 80 of \cite{MR0827219}),  that the stabilizer of such a pair $\{C', D-C'\}$ in $\mathord{\rm HS}$ is ${\PSU}(3,\F_{25})\rtimes \Z/2\Z$ with index 176.  
Therefore the point-wise stabilizer of $\{A, B, C, D\}$ is ${\PSU}(3,\F_{25})$.
We consider ${\PSU}(3,\F_{25})$ as a subgroup of ${\OG}(U\oplus \Lambda)$ acting trivially on $U$.
The group ${\PSU}(3,\F_{25})$ preserves the projection $w'$ of the Weyl vector $w$ which is conjugate to an ample class of $X$
(Lemma~\ref{weyl3}).
On the other hand, ${\PSU}(3,\F_{25})$ acts on $R_1$ identically, and hence acts trivially on $R_1\dual/R_1 \cong S_X\dual/S_X$.  
This implies that ${\PSU}(3,\F_{25})$ preserves the period of $X$.
It now follows from the Torelli theorem for supersingular $K3$ surfaces due to Ogus~\cite{MR563467, MR717616} that ${\PSU}(3,\F_{25})$ 
can act on $X$ by automorphisms.
\end{proof}
\begin{remark}
By the direct calculation using the data of Section~\ref{subsec:D0} and~\eqref{eq:barg},
we can confirm that the image of $\Aut(X, D_0)$  by the natural homomorphism $\OG(S_X)\to \OG(q_{S_X})$
is equal to~\eqref{eq:six}, and hence is of order $6$.
Combining this fact with 
the proof of Proposition~\ref{group3},
we see that 
the kernel of $\Aut(X, D_0)\inj \OG(S_X) \to \OG(q_{S_X})$ is isomorphic to the simple group  ${\PSU}(3,\F_{25})$.
\end{remark}
\subsection{Proof of Theorem~\ref{thm:main}\,(2)}
Next we consider the following vectors in the Leech lattice $\Lambda$: 
\begin{equation}\label{XYZ2}
A = 4\nu_{\infty} + \nu_{\Omega},\;\; 
B = 0, \;\;
C = 2\nu_{K_0},\;\;
  D = \nu_{\Omega} - 4\nu_{\infty},
\end{equation}
where $K_0$ is an octad which does not contain $\infty$.
Consider  the vectors in $L = U \oplus \Lambda$ defined by
\begin{equation}\label{xyz2}
a = -(2, 1, A),\;\;  
b = (-1, 1, 0), \;\;
 c = (0, 1, C), \;\;
   d = (0,0, D).
\end{equation}
Obviously we have 
\begin{eqnarray*}
&& a^2= b^2 = -2, \;\; c^2 = d^2 =-4, \;\; \langle a, b \rangle = \langle b, c \rangle = -1, \\
&& \langle a, c \rangle = \langle b, d \rangle = 0, \;\;
 \langle a, d \rangle = 1, \;\;
 \langle c, d \rangle = -2.
\end{eqnarray*}
Let $R_2$ be the sublattice of $L$ generated by $a, b, c, d$.  
Note that the Gram matrix of 
$R_2$ is the same as the one given in~\eqref{eq:GramR}.
%Moreover $R^*/R$ is generated by
%$${-3x - z - t\over 5}, \ {-x-z + y -2t\over 5}.$$
Moreover the alternating group $\mathfrak{A}_8$ of degree 8 acts on the set $\Omega = \{ \infty, 0, 1,..., 22\}$ such that it preserves the octad $K_0$ and fixes
the point $\infty$ (see Conway~\cite{MR0338152}).  This action can be extended to the one on 
$\Lambda$, and hence on $L = U\oplus \Lambda$
acting trivially on $U$.  
By definition, $\mathfrak{A}_8$ fixes $R_2$.
Let $S_2$ be the orthogonal complement of $R_2$ in $L$, 
on which $\mathfrak{A}_8$ acts.
Then $S_2$ is isomorphic to the N\'eron-Severi lattice $S_X$ of the supersingular $K3$ surface $X$ with Artin invariant $1$ in characteristic $5$.
\begin{lemma}\label{weyl}
Let $w'$ be the projection of the Weyl vector $w$ into $S_2\dual$.  Then $5w' \in S_2$ and $(5w')^2=60$.  
Moreover $5w'$ is conjugate to the class of an ample divisor on $X$
under the action of $\Wgr (S_2)$.
\end{lemma}
\begin{proof}
Write $w = w' + w''$ where $w''$ is the projection of $w$ into $R_2\dual$.
We  see that $w'' = (6a -5b-c +2d)/5$ and $(w'')^2 = -12/5$.  
Since $5w'' \in R_2$ and $w^2 = 0$, we have $5w' \in S_2$ and $(w')^2 = 12/5$.
The proof of the last assertion is the same as that of Lemma~\ref{weyl3}.
\end{proof}
\begin{lemma}\label{168}
There exist exactly $168$ Leech roots which are orthogonal to $R_2$, and $\mathfrak{A}_8$ acts transitively on these Leech roots.
\end{lemma}
\begin{proof}
By an argument similar to the proof of Lemma~\ref{252}, 
we  see that
the desired Leech roots correspond to $(-4)$-vectors 
$$
4\nu_{\infty} + \nu_{\Omega} - 2\nu_K
$$ 
in $\Lambda$, where
$K$ are  octads which satisfy $K \ni \infty$ and $|K \cap K_0| = 2$.  We count the number of such octads $K$.
Let $a_1, a_2$ be in $K_0$.  Then the number of octads containing three points $\infty, a_1, a_2$ is 21 (see Conway~\cite{MR0338152}, Theorem 11).
Take two points $a_3, a_4 \in K_0\setminus \{a_1, a_2\}$.  Then there exists exactly one octad containing $5$ points $\infty, a_1, a_2, a_3, a_4$.
Thus the number of octads $K$ containing $\infty, a_1, a_2$ and satisfying $K\cap K_0 =\{a_1, a_2\}$ is $21 - {6\choose 2}  = 6$.  Therefore
the number of octads $K$ containing $\infty$ and satisfying $|K\cap K_0| = 2$ is ${8\choose 2}\times 6 = 168$.
\par
Now take  such an octad $K$.  Then
the stabilizer subgroup of $K$ in $\mathfrak{A}_8$ is the symmetry group $\mathfrak{S}_5$ of degree 5 because it has five orbits of size $1, 2, 5, 6, 10$;
that is, 
$$
\{ \infty\}, \ \{ K\cap K_0 \}, \ \{ K_0\setminus ((K\cap K_0)\cup \{\infty\})\}, \{ K\setminus (K\cap K_0)\}, \ 
\{\Omega \setminus (K\cup K_0)\}.
$$
Since the index of $\mathfrak{S}_5$ in $\mathfrak{A}_8$ is 168, we have the second assertion.
\end{proof}

\begin{lemma}%\label{} 
The group $\mathfrak{A}_8$ acts on $X$ by automorphisms.
\end{lemma}
\begin{proof}
The proof is similar to that of Lemma~\ref{group3}.
\end{proof}
\noindent
Finally the $168$ Leech roots are the classes of the $168$ smooth rational curves on $X$ because
Leech roots have the minimal degree 1 with respect to the Weyl vector $w$.
Thus we have finished the proof of Theorem~\ref{thm:main}\,(2).
\begin{remark}
Let $r=(1,1,4\nu_{\infty} + \nu_{\Omega} - 2\nu_K), \ 
r'=(1,1,4\nu_{\infty} + \nu_{\Omega} - 2\nu_{K'})$ be distinct two Leech roots in
Lemma~\ref{168}.  Then $\langle r, r'\rangle = 0$ or $1$ if and only if
 $|K\cap K'| =4$ or $2$
respectively.  Moreover we see that  
there exist exactly $72$ Leech roots $r'$ in Lemma~\ref{168} with $\langle r, r'\rangle =1$ (see Proposition~\ref{Shimada168}).
\end{remark}
\begin{remark}\label{weyl-4}
In both cases (1) and (2) in Theorem~\ref{thm:main}, the octads $K$ satisfying $\infty \in K$ and $|K\cap K_0| = 2$ appear.
In case (1), $K$ satisfies one more condition that $K$ does not contain $0$.
Here we discuss the remaining octads $K$; that is, $K$ contains $\infty, 0$ and satisfies $|K\cap K_0|=2$.
We put 
$$
r = (2,2, \lambda), \ \lambda = 2\nu_K + \nu_{\Omega} - 4\nu_0,
$$
where $K$ is an octad with $K \ni \infty$, $K\ni 0$ and $| K \cap K_0| =2$.  
Then $r^2 = -2$ and
$r \in R_1^{\perp} = S_1$.  
Obviously we have $\langle r, w' \rangle = \langle r, w \rangle = 2$.
There exist exactly 42 octads $K$ satisfying $K \ni \infty$, $K\ni 0$ and $| K \cap K_0| =2$.
Recall that $w' = (2,2, A)=(2,2,4\nu_{\infty}+\nu_{\Omega})$ (Remark \ref{weyl-3}).
%
%\par
For each root $r$ from the above $42$ roots, put
$$r' = 2w' - r = (2,2, 8\nu_{\infty} + 4\nu_0 + \nu_{\Omega} - 2\nu_K).$$
Then $(r')^2 =-2$ and $r' \in R_1^{\perp} = S_1$.  Thus the class $r + r'$ corresponds to the pullback of a conic
on ${\P}^2$ tangent to the Fermat sextic $C_F$ at six points (see Proposition~\ref{Shimada42}).
\end{remark}
\subsection{Proof of Theorem~\ref{thm:main}\,(3)}
Finally we consider the following vectors in the Leech lattice $\Lambda$: 
\begin{equation}\label{XYZ3}
\renewcommand{\arraystretch}{1.4}
\begin{array}{l}
A = 4\nu_{\infty} + \nu_{\Omega},\;\;
B = 0,\;\; 
C = 8\nu_{\infty}, \\
D = 2(\nu_{\infty} + \nu_0 +\nu_1 +\nu_2) - 2(\nu_{3} +
\nu_5+\nu_{14} + \nu_{17}).
\end{array}
\end{equation}
Here $K_0 =\{\infty, 0, 1, 2, 3, 5, 14, 17\}$ is an octad (see Todd~\cite{MR0202854}).
Consider  the vectors in $L = U \oplus \Lambda$ defined by
\begin{equation}\label{xyz3}
a = -(2, 1, A),\;\;  b = (-1, 1, 0), \;\;  c= (1, 2, C), \;\; d = (0,0, D).
\end{equation}
Obviously we have 
\begin{eqnarray*}
&&a^2= b^2 = -2, \;\; c^2 = d^2 =-4, \;\; \langle a, b \rangle = \langle b, c \rangle = -1,\\
&&\langle a, c \rangle = \langle b, d \rangle =0, \;\; \langle a, d \rangle = 1, \;\;
\langle c, d \rangle =-2.
\end{eqnarray*}
Let $R_3$ be the sublattice of $L$ generated by $a, b, c, d$.  
Then the Gram matrix of $R_3$ is the same as the one given in~\eqref{eq:GramR}.
Note that a subgroup group $(\Z/2\Z)^4 \rtimes (\Z/3\Z \times \mathfrak{S}_4)$ of $M_{23}$ acts on the set $\Omega = \{ \infty, 0, 1,..., 22\}$ such that it preserves 
the sextet of tetrads determined by $\{\infty, 0, 1, 2\}$, preserves the set $\{ 0,1,2\}$
and the octad $K_0$, and fixes
the point $\infty$ (see Conway~\cite{MR0338152}).  This action can be extended to the one on 
$\Lambda$, and hence on $L = U\oplus \Lambda$
acting trivially on $U$.  
Let $S_3$ be the orthogonal complement of $R_3$ in $L$.
Then $S_3$ is isomorphic to the N\'eron-Severi lattice $S_X$ of the supersingular $K3$ surface $X$ with Artin invariant $1$ in characteristic $5$.
\begin{lemma}\label{weyl2}
Let $w'$ be the projection of the Weyl vector $w$ into $S_3\dual$.  
Then $5w' \in S_3$ and $(5w')^2=80$.  
Moreover $w'$ is conjugate to the class of an ample divisor on $X$ under the action of $\Wgr (S_3)$.
\end{lemma}
\begin{proof}
Write $w = w' + w''$ where $w'' \in R_3\dual$.  Then $w'' = (6a -4b-3c +3d)/5$ and $(w'')^2 = -16/5$.  
Since $5w'' \in R_3$ and $w^2 = 0$, we have  $5w' \in S$ and $(w')^2 = 16/5$.
The proof of the last assertion is the same as that of Lemma~\ref{weyl3}.
\end{proof}
\begin{lemma}\label{48}
There exist exactly $96$ Leech roots which are orthogonal to $R_3$.
\end{lemma}
\begin{proof}
By an argument similar to the proof of Lemma~\ref{252}, 
we  see that 
the desired Leech roots are 
$$
(1,1,A- 2\nu_K),
$$ 
where $K$ is an octad satisfying
one of the following conditions: 
\begin{itemize}
\item[(1)]
$|K \cap K_0| = 4$, $K\ni \infty$ and $K$ contains exactly 
two points of $\{0,1,2\}$,
\item[(2)] $|K \cap K_0| = 2$, $K\ni \infty$ and $K$ contains exactly 
one point of $\{0,1,2\}$.
\end{itemize}
We count the number of octads satisfying (1) or (2).
In case (1), 
there are $21$ octads containing fixed three points $\{\infty, 0, 1\}$ and among these 21 octads, five octads contain four points $\{\infty, 0,1,2\}$.
Thus for each two points from $\{ 0,1,2\}$, there exist exactly 16 octads, and the total is $16 \times 3 = 48$.  
In case (2), there are exactly $16$ octads $K$ satisfying  $K\cap K_0 = \{\infty, 0\}$
(see Conway~\cite{MR0338152}, Table 10.1).  Thus we have $48$ octads satisfying the condition (2).
\end{proof}

\begin{lemma}%\label{} 
The group $(\Z/2\Z)^4 \rtimes (\Z/3\Z \times \mathfrak{S}_4)$ acts on $X$ by automorphisms.
\end{lemma}
\begin{proof}
The proof is similar to that of Lemma~\ref{group3}.
\end{proof}
\noindent
The 96 Leech roots are the classes of the 96 smooth rational curves on $X$ because
Leech roots have the minimal degree 1 with respect to the Weyl vector $w$.
Thus we have finished the proof of Theorem~\ref{thm:main}\,(3).
\par
\medskip
We denote by ${\TTT}$ the set of $96$ Leech roots in Lemma~\ref{48}.  
Let ${\TTT}_{ij}$ be the set of Leech roots which correspond to
the octads $K$ containing the two point $i, j$ ($i, j = 0,1,2)$ in the proof of Lemma~\ref{48}, case (1), 
and let ${\TTT}_i$ be the set
of all Leech roots corresponding to the octads $K$ 
containing the point $i$ ($i=0,1,2$) in the proof of Lemma~\ref{48}, case (2).
\begin{theorem}\label{48-16-3}
Each ${\TTT}_i$, ${\TTT}_{ij}$ consists of $16$ mutually orthogonal Leech roots.
Each Leech root  in ${\TTT}_i$ {\rm (}resp. ${\TTT}_{ij}${\rm)}  meets exactly $6$ Leech roots 
in ${\TTT}_j$ with $j\ne i$ 
{\rm (}resp. ${\TTT}_{kl}$ with  $(k,l)\ne  (i,j)${\rm)} with multiplicity $1$.
In particular, $\{{\TTT}_i, {\TTT}_j\}$ and $\{{\TTT}_{ij}, {\TTT}_{kl}\}$ form 
a $(16_6)$-configuration.  Moreover $\{ {\TTT}_i , {\TTT}_{jk} \}$  with $\{i,j,k\} = \{0,1,2\}$ 
is a $(16_{12})$-configuration and $\{ {\TTT}_i , {\TTT}_{ij} \}$ is a $(16_{4})$-configuration.
\end{theorem}
\begin{proof}
We put  $r =(1,1, A-2\nu_K)$ and $ r' =(1,1,A- 2\nu_{K'}) \in {\TTT}$.  Then 
$\langle r, r'\rangle = 0$ or $1$ if and only if 
$|K\cap  K'| = 4$ or $2$, respectively.  Since any two octads meet at $0, 2$ or $4$ points, 
${\TTT}_{ij}$ consists of $16$ mutually orthogonal Leech roots.   

On the other hand, if $r, r' \in {\TTT}_i$ and $K\cap K' = \{ \infty, i\}$, then
the symmetric difference $K+K'$ and $\Omega + K + K'$ are dodecads.  Note that
$\Omega + K + K'$ contains the octad $K_0$.  This contradicts the fact that no dodecads
contain an octad.  Thus we have $|K\cap K'| = 4$, and hence ${\TTT}_i$ consists of
$16$ mutually disjoint Leech roots.

Finally we  see that an element from ${\TTT}_i$ or ${\TTT}_{ij}$ has the
incidence relation with ${\TTT}_j$ and ${\TTT}_{kl}$ as desired.  Since
the group $(\Z/2\Z)^4 \rtimes (\Z/3\Z \times \mathfrak{S}_4)$ acts transitively on each set ${\TTT}_i, {\TTT}_{ij}$, the assertion follows.
\end{proof}
By defining $\{\SSS_{ij}\}$ by 
$$\SSS_{01}={\TTT}_0,\;\; \SSS_{02}={\TTT}_1,\;\; \SSS_{03}={\TTT}_2,\;\;
\SSS_{11}={\TTT}_{12},\;\; \SSS_{12}={\TTT}_{02},\;\; \SSS_{13}={\TTT}_{01},$$ 
we have finished the proof of Theorem~\ref{thm:six}.

%%%%%%%%%%

%
\section{Supersingular elliptic curve in characteristic $5$}\label{sec:E}
We summarize some facts
on the supersingular elliptic curve in characteristic $5$
which we will use later.  
We have, up to isomorphisms, only one supersingular elliptic curve 
defined over an algebraically closed field $k$ of characteristic $5$, which is given by the equation
$$
               y^{2}  = x^{3} -1.
$$
We denote by $E$ a nonsingular complete model of the supersingular 
elliptic curve. 
In the affine model, let $(x_{1}, y_{1})$ and $(x_{2}, y_{2})$ be
two points on $E$. Then, the addition 
$$
m : E\times E\to E
$$
of $E$ is given by
\begin{equation}\label{eq:add}
\renewcommand{\arraystretch}{2}
\begin{array}{l}
m^* x = -x_{1} - x_{2} + \dfrac{(y_{2} - y_{1})^2}{(x_{2} - x_{1})^{2}},\\
m^* y = y_{1} + y_{2} - \dfrac{(y_{2} - y_{1})^3}{(x_{2} - x_{1})^3} + 
\dfrac{3(x_{2}y_{1} - x_{1}y_{2})}{(x_{1} - x_{2})}.
\end{array}
\end{equation}
We denote by $[n]_{E}$ the multiplication by an integer $n$, and by $E_{n}$
the group of $n$-torsion points of $E$. % and we set $A_{n} = \Ker [n]_{E}$.  
The multiplication $[2]_{E}$ is concretely
given by
$$
[2]_{E}^*\, x = x_{1} + 1/y_{1}^{2},\;\;\;\; [2]_{E}^*\, y = 2y_{1} - 1/y_{1} + 1/y_{1}^{3}.
$$
%$$
%\begin{array}{l}
%    x = x_{1} + 1/y_{1}^{2},\\
%    y = 2y_{1} - 1/y_{1} + 1/y_{1}^{3}.
%\end{array}
%$$
We denote by ${\Fr}$ the relative Frobenius morphism. Then, it satisfies
$$
%{\Fr}^{2} = -5.
{\Fr}^{2} = [-5]_E.
$$
We set $\omega = 2 + 3\sqrt{2}$. Then, $\omega$ is a primitive cube root of unity.
We set 
$$
%P_{\infty} = (0,1,0), P_{0} = (1,0,1), P_{1} = (\omega,0,1), P_{2} = (\omega^2, 0, 1).
P_{\infty} = (0, \infty),\;\;  P_{0} = (1,0), \;\;  P_{1} = (\omega,0), \;\;  P_{2} = (\omega^2, 0).
$$
The point $P_{\infty}$ is the zero point of $E$, and the group $\KerE{2}$ of 
$2$-torsion points of $E$ is  %given  as follows:
$$
\KerE{2} = \{P_{\infty}, P_{0}, P_{1}, P_{2}\}.
$$
The translation $T_{P_{0}}$ by the point $P_{0}$ is given by
$$
 T_{P_{0}}^{*}(x) = \dfrac{x + 2}{x - 1},\quad
 T_{P_{0}}^{*}(y) = \dfrac{2y}{(x - 1)^2}.
 $$
%$$
%\begin{array}{l}
%     T_{P_{0}}^{*}(x) = \dfrac{x + 2}{x - 1}, \\
%     T_{P_{0}}^{*}(y) = \dfrac{2y}{(x - 1)^2}.
%\end{array}
%$$
We set
$$
u = 2(x + T_{P_{0}}^{*}(x) -1),\;\; v =   2\sqrt{2}(y + T_{P_{0}}^{*}(y)).
$$
%$$
%\begin{array}{l}
%    u = 2(x + T_{P_{0}}^{*}(x) -1),\\
%   v =   2\sqrt{2}(y + T_{P_{0}}^{*}(y)).
%\end{array}
%$$
Then, $u$ and $v$ are invariant under the action of $T_{P_{0}}^{*}$,
and we have
$$
u = \dfrac{2x^2  +3 x + 1}{(x -1)},\;\;
v =   \dfrac{2\sqrt{2}y(x^2 +3x + 3)}{(x - 1)^2}.
$$
%$$
%\begin{array}{l}
%    u = \dfrac{2x^2  +3 x + 1}{(x -1)},\\
%    v =   \dfrac{2\sqrt{2}y(x^2 +3x + 3)}{(x - 1)^2}.
%\end{array}
%$$
We know that the degree of the field extension  $k(x, y)/k(u, v)$ is equal to 2
and that $u$ and $v$ satisfy the equation $v^2 = u^3 - 1$.
Therefore, we have the quotient morphism by the action of $T_{P_{0}}$:
$$
\begin{array}{rccc}
  \phi_{E, 2} :& E & \to & E \\
      & (x, y) & \mapsto  & (u, v).
\end{array}
$$
By a direct calculation, we see that
$$
    \phi_{E, 2}^{2} =  [-2]_{E}.
$$
The elliptic curve $E$ has the following automorphism $\gamma$ of order $6$ defined by 
$$
\gamma^*x=\omega x, \quad \gamma^* y=-y.
$$
%$$
%\begin{array}{l}
%  \gamma : x\mapsto \omega x,~y \mapsto  -y. \\
%\end{array}
%$$
%which satisfies $\gamma^{6} = {\id}$.
%$$
%     \gamma^{6} = {\id}.
%$$
We consider the endomorphism ring $\OOO={\End}(E)$. We set
$B  = {\End}(E)\otimes_{\Z} {\Q}$. 
Then, as is well-known,
$B$ is the quaternion division algebra with discriminant $5$ and  $\OOO$
is a maximal order of $B$. We consider the following elements of $\OOO$:
$$
\omega_{1} = 1, \;\; ~\omega_{2} = \gamma, \;\; ~\omega_{3} =  \phi_{E, 2}, \;\;  ~\omega_{4} =\gamma  \phi_{E, 2}.
$$
The multiplication is given as follows:
\par
\medskip
\begin{center}
\begin{tabular}{| l |c|c|c|c|}
\hline
     & $\gamma $&$ \phi_{E, 2}$& $\gamma  \phi_{E, 2}$ \\
\hline
$\gamma $ & $\gamma -1$ & $\gamma  \phi_{E, 2}$ & $- \phi_{E, 2} +\gamma  \phi_{E, 2}$\\
\hline
$ \phi_{E, 2}$ & $-1 + \phi_{E, 2} - \gamma  \phi_{E, 2}$ & $-2$ & $-2 + 2\gamma- \phi_{E, 2}$  \\
\hline
$\gamma  \phi_{E, 2}$ & $-\gamma  + \phi_{E, 2}$ & $-2\gamma$ & $-2 - \gamma  \phi_{E, 2}$ \\
\hline
\end{tabular}
\end{center}
\par
\medskip\noindent
For example, we have $\phi_{E, 2}\gamma= -1 + \phi_{E, 2} - \gamma  \phi_{E, 2}$.

The canonical involution $a\mapsto \bar{a}$ of the quaternion algebra $B$ is given as follows:
$$
\bar{\gamma} = - \gamma^{2},
\;\;
\overline{\phi_{E, 2}} = - \phi_{E, 2},
\;\;
\overline{\gamma \phi_{E, 2}} = -1 - \gamma  \phi_{E, 2} .
$$
Denoting by ${\Tr}$ the trace map in $B$, we have a $4\times 4$  matrix 
$({\Tr}~ \omega_{i}\omega_{j})$:
$$
\left[
\begin{array}{cccc}
  2 & 1 & 0 & 1 \\
1 & -1 & -1 & -1 \\
0 & -1 & -4 & -2 \\
-1 & -1 & -2 & -3 
\end{array}
\right].
$$
Since the determinant of this matrix is equal to $-25$, we know that
$\omega_{i}~(i = 1, 2, 3, 4)$ is a basis of the maximal order $\OOO$:
$$
      \OOO = {\Z} + {\Z}\gamma + {\Z}\phi_{E, 2}+ 
{\Z}\gamma \phi_{E, 2}.
$$
 \begin{remark}\label{Frobenius}
Considering $\Ker ({\Fr } - 1) = E({\F}_{5}) \cong {\Z}/6{\Z}$ , we have
$$
{\Fr }  =   1+\phi_{E, 2}\gamma (1 + \gamma)=-1 + \phi_{E, 2} - 2\gamma \phi_{E, 2}.
$$
\end{remark}

%%%%%%%%%%%
%
\section{Number of ${\F}_{p^2}$-rational points on ${\Km}(A)$}\label{sec:Numb}
%
%Let $k$ be a field of characteristic $p > 2$,  and 
Let $E$ be a supersingular
elliptic curve defined over ${\F}_{p}$. 
We set $A = E \times E$ and
denote by $\iota_A$ the inversion of $A$. 
We denote by
${\Km}(A)$ the Kummer surface associated with $A$.
%which is obtained by the resolution
%of singularities of the quotient surface $A/\langle \iota_A \rangle$
%by the group $\langle \iota_A \rangle$. 
%We denote by $N$ the number of ${\F}_{p^2}$-rational points on ${\Km}(A)$.
In this section,
we compute the number $N$ of ${\F}_{p^2}$-rational points on ${\Km}(A)$.

In Katsura~and~Kondo~\cite{MR2862188}, we proved the following lemma.
For the readers' convenience, we give here the proof again.
\begin{lemma}
$E({\F}_{p^2}) = {\Ker}[p + 1]_{E}$. In particular,  we have
$| E({\F}_{p^2})| = (p + 1)^{2}$ and 
$| A({\F}_{p^2})| = (p + 1)^{4}$.
\end{lemma}
\begin{proof}
A point $P \in E$ is contained in $E({\F}_{p^{2}})$ if and only if 
${\Fr}^{2}(P) = P$. Since ${\Fr}^{2} = [-p]_E$, we have 
${\Fr}^{2}(P) = P$ if and only if $[p + 1]_{E}(P) = 0$.
\end{proof}
\begin{theorem}
The number $N$ of ${\F}_{p^2}$-rational points on ${\Km}(A)$
is equal to $1 + 22p^2 + p^4$.
\end{theorem}
\begin{proof}
We consider the quotient morphism
$$
 \varpi : A \to A/\langle \iota_A \rangle.
$$
By ${\Ker}[2]_{A} \subset  {\Ker}[p + 1]_{A}$,
all $2$-torsion points are defined over $\F_{p^2}$.
Excluding the 2-torsion points, we get 
$\{(p + 1)^4 - 16\}/2$  points 
of ${\Km}(A)(\F_{p^2})$ derived from $(p + 1)$-torsion points 
on $A$. If a point $P$ on $A$ satisfies $\Fr^2(P) = \iota_A(P)$, then
we have $\Fr^{2}( \varpi (P)) = \varpi (P)$ on $A/\langle \iota_A \rangle$.
Therefore, $ \varpi (P)$ is an ${\F}_{p^2}$-rational point 
on $A/\langle \iota_A \rangle$. Hence, it gives
an ${\F}_{p^2}$-rational point on ${\Km}(A)$.
Since $\Fr^2(P) = \iota_A(P)$ holds if and only if $P$ is contained in $  {\Ker}[p - 1]_{A}$,
the number of such points on $A$ is equal to $(p -1)^4$.
Excluding the 2-torsion points, we get
$\{(p - 1)^4 - 16\}/2$  points 
of ${\Km}(A)(\F_{p^2})$ derived from $(p - 1)$-torsion points on $A$.
Since %on ${\P}^1$ we have  $p^{2} + 1$ ${\F}_{p^2}$-rational points,
$|\P^1(\F_{p^2})|=p^2+1$,
we have $16 (p^{2} + 1)$ points of ${\Km}(A)(\F_{p^2})$
that come from the 16 exceptional curves.
Therefore, in total we have an inequality
$$
N\ge \{(p + 1)^4 - 16\}/2 + \{(p - 1)^4 - 16\}/2 + 16 (p^{2} + 1)=1 + 22p^2 + p^4.
$$
%$$
%\begin{array}{cl}
%   N  & \geq \{(p + 1)^4 - 16\}/2 + \{(p - 1)^4 - 16\}/2 + 16 (p^{2} + 1)\\
%
%   & = 1 + 22p^2 + p^4
%\end{array}
%$$
On the other hand, we consider the congruent zeta function 
$Z({\Km}(A)/{\F}_{p^2}, t)$
of ${\Km}(A)$. Since ${\Km}(A)$ is a $K3$ surface, we have
$$
Z({\Km}(A)/{\F}_{p^2}, t) = \left((1 - t) (1 - p^4t)\prod_{i = 1}^{22}(1 -\alpha_{i}t)\right)\inv
$$
with algebraic integers $\alpha_{i}$ %$(i = 1,  \ldots , 22)$
satisfying $| \alpha_{i} | = p^2$.
Since  $\log Z({\Km}(A)/{\F}_{p^2}, t) = Nt + \cdots $, we have
$$
  N = 1 + \sum_{i = 1}^{22}\alpha_{i} + p^4 \leq 1 + \sum_{i = 1}^{22}| \alpha_{i}|  + p^4 = 1 + 22p^2 + p^4.
$$
%We know $| \alpha_{i} | \leq p^2$ ($i = 1, 2, \cdots , 22$). Therefore, 
%we have
%an equality
%$$
%   N \leq 1 + \sum_{i = 1}^{22}| \alpha_{i}|  + p^4 = 1 + 22p^2 + p^4.
%$$
Hence, we have $N = 1 + 22p^2 + p^4$.
\end{proof}
\begin{corollary}
If $p = 5$, we have $| {\Km}(A)({\F}_{25})| = 1176$.
\end{corollary}
\begin{remark}\label{rational points}
Let $E$ be the nonsingular complete model  of the supersingular elliptic curve 
defined by $y^{2} = x^3 -1$ in characteristic 5. Then,  by the consideration above, 
a point $P =(a, b) \in E$ is contained in $\KerE{4}\setminus \KerE{2}$
if and only if ${\Fr}^2(P) = -P$ and $b \neq 0$. Therefore, we have the following.
\begin{itemize}
\item[\rm (i)] $P \in \KerE{2}$ if and only if $b = 0$. (Hence, $a \in {\F}_{25}$);
\item[\rm (ii)] $P \in \KerE{4}\setminus \KerE{2}$ if and only if $a\in {\F}_{25}$ and $b \not\in {\F}_{25}$;
\item[\rm (ii)] $P \in \KerE{6}\setminus \KerE{2}$ if and only if $a\in {\F}_{25}$ and $b \in {\F}_{25}\setminus \{0\}$.
\end{itemize}
\end{remark}

%%%%%%%%%%%
%
\section{Six sets of disjoint $16$ smooth rational curves on $\Km(A)$}\label{sec:Kum}
In this section,
we resume working in characteristic $5$.
Let $E$ be the elliptic curve defined by $y^2=x^3-1$,
and let $A$ be the abelian surface $E\times E$.
For brevity, we denote by $Y$ the Kummer surface $\Km(A)$.
As is well-known~(see Ogus~\cite{MR563467}), $Y$ is isomorphic to our supersingular $K3$ surface $X$
with Artin invariant 1.
In this section, we explicitly construct six sets
$$
\SSS_{00}, \SSS_{01}, \SSS_{02}, \SSS_{10}, \SSS_{11}, \SSS_{12}
$$
of disjoint $16$ smooth rational curves on $Y$
with the properties (a), (b), (c) in Theorem~\ref{thm:six},
and prove Theorem~\ref{cor:six}.
We denote by $\SA$ and $\SY$ the N\'eron-Severi lattices of $A$ and $Y$,
respectively.
It is well-known that $\SA$ is of discriminant $-25$.
\par
\medskip
We denote by $A_2$ the group of $2$-torsion points of $A$:
$$
A_2=E_2\times E_2.
$$
%For $n\in \Z_{>0}$, 
%we denote by $A_n$ the kernel of the homomorphism $[n]_A: A\to A$:
%$$
%A_n =\KerE{n}\times \KerE{n}.
%$$
%
We consider the following commutative diagram:
$$
\renewcommand{\arraystretch}{1.2}
\begin{array}{ccc} 
\tilA & \maprightsp{\pi} & Y \\
\mapdownleft{b} & & \mapdownright{\rho}\\
A &\maprightsb{\varpi} &A/\gen{\iota_A},
\end{array}
$$
where  $b$ is the blow-up at the points of $A_2$,
$\varpi$ is the quotient morphism by $\gen{\iota_A}$,
$\rho$ is the minimal resolution, and $\pi$ is the  double covering induced by $\varpi$.
For $P\in A_2$, we denote by $E_P$ the exceptional curve of $b$ over $P$.
The homomorphism $b^* :\SA\to \StilA$  identifies $\SA$ with a sublattice of 
 the N\'eron-Severi lattice $\StilA$ of $\tilA$,
and we obtain an orthogonal decomposition 
\begin{equation}\label{eq:orthoStilA}
\StilA=\SA\oplus \bigoplus_{P\in A_2} \Z [E_P].
\end{equation}
Let $\TTT$ denote the group of translations of $A$ by the points in $A_2$.
Then $\TTT$ acts on $\tilA$,
and hence on $\StilA$.
The action  preserves the orthogonal decomposition~\eqref{eq:orthoStilA},
and its restriction to the factor $\SA$ is trivial,
while its restriction to the factor $\bigoplus \Z [E_P]$ is induced by 
the permutation representation of $\TTT$ on $A_2$.
The inversion $\iota_A$ of $A$ lifts to an involution $\tilde\iota_A$ of $\tilA$,
and $\pi$ is the quotient map by $\gen{\tilde\iota_A}$.
The homomorphism $\pi^*$  induces an embedding of  the lattice $\SY(2)$ into $\StilA$,
where $\SY(2)$ is the $\Z$-module $\SY$ with the symmetric bilinear form defined by 
$\intM{x, y}{\SY(2)} =2\intM{x, y}{\SY}$.
\par
For an irreducible curve $\varGamma$ on $A$ that is invariant under $\iota_A$,
we denote by $\varGamma_{\tilA}$ the \emph{strict} transform of $\varGamma$ by  $b:\tilA\to A$,
and by $\varGamma_Y$ the image of $\varGamma_{\tilA}$ by $\pi:\tilA\to Y$ 
with the reduced structure.
Since $\varGamma$  is invariant under $\iota_A$,
the map $\pi$ induces a double covering $\varGamma_{\tilA} \to \varGamma_Y$.
Suppose that $\varGamma$ is smooth.
Then we have
$$
[\varGamma_{\tilA}]=[b^{*}\varGamma]-\sum_{P\in \varGamma\cap A_2} [E_P].
$$
\par
For an endomorphism $g: E\to E$ of $E$,
we denote by
$\Phi_g$ the  graph  of $g$; that is,
$$
\Phi_g =\set{(P, g(P))}{P\in E}.
$$
We can calculate the intersection number of a  curve of certain type on $A$ with $\Phi_g$
by the following method.
%let $E$ be the nonsingular complete model of a supersingular elliptic curve 
%defined by the equation $y^2 = x^3  - 1$ and 
Suppose that $H$ is a (hyper-)elliptic curve  defined by
$$
v^2=f_H(u),
$$
with the  involution $\iota_H: (u, v)\mapsto (u, -v)$.
We consider two finite morphisms 
$$
\eta_{i} : H \to E \;\;(i = 1, 2)
$$
 satisfying  $\eta_i\circ \iota_H=\iota_E\circ \eta_i$,
and  we set 
$$
\eta = (\eta_{1}, \eta_{2}) : H \to E \times E = A.
$$
We denote by $\sGamma [\eta]$  the image of $\eta$ on $A$ 
with the reduced structure.
Suppose that $\eta$ induces a birational map from 
$H$ to $\sGamma[\eta]$.
Using the addition $m : E \times E \to E$, we have a divisor
$$
\Delta = \Ker m = \{(P, -P) \mid P \in E\}
$$
on $A=E\times E$.
From the given endomorphism $g \in {\rm End}(E)$, we obtain a morphism
$$
      (-g)\times \id :  E \times E \to  E \times E.
$$
Then we have 
$\Phi_{g} = ((-g)\times \id)^{*}\Delta$.
We consider the morphism
$$
\begin{array}{lccccccc}
\theta : &   H & \maprightsp{\eta} &E \times E & 
\maprightsp{(-g)\times \id} &  E \times E & \maprightsp{m} & E.
\end{array}
$$
Then we have 
\begin{equation}\label{eq:degtheta}
\begin{array}{ll}
\intM{\sGamma [\eta], \Phi_g}{\SA} & =\deg \eta^{*}\Phi_g = \deg (\eta^{*}\circ ((-g)\times{\rm \id})^{*}\Delta)\\
  &= \deg (\eta^{*}\circ ((-g)\times{\rm  id})^{*}\circ m^{-1}(P_{\infty})) \\
  &=\deg ((m\circ ((-g)\times \id) \circ\eta)^{*}(P_{\infty})),\\
& = \deg \theta. % = (- \ord_{P_{\infty}}\theta^{*}x)/2.
\end{array}
\end{equation}
By the assumption $\eta_i\circ \iota_H=\iota_E\circ \eta_i$, 
the map $\eta_i$ is written as
$$
\eta_i^* x= M_i(u),
\quad
\eta_i^* y= v\cdot N_i(u),
$$
by some rational functions $M_i$ and $N_i$ of one variable $u$.
Since $g: E\to E$ satisfies $g\circ \iota_E=\iota_E\circ g$,
there exist rational functions $\Psi$ and $\Xi$ of one variable $x$
such that 
$$
g^* x= \Psi(x),
\quad
g^* y= y\cdot\Xi(x).
$$
The morphism $\theta$ induces a finite morphism
$$
\tilde{\theta}: H/\gen{\iota_H}=\P^1\to E/\gen{\iota_E}=\P^1
$$
from the $u$-line to the $x$-line.
Using~\eqref{eq:add}, 
we see that  $\tilde{\theta}$ is given by the rational function
$$
\tilde{\theta}^*x=-\Psi(M_1(u))-M_2(u)+\frac{f_H(u) \cdot (N_2(u) + N_1(u) \cdot \Xi(M_1(u)))^2}{(M_2(u)-\Psi(M_1(u)))^2}.
$$
Since $\deg \tilde{\theta}=\deg \theta$,
we can calculate $\intM{\sGamma [\eta], \Phi_g}{\SA}=\deg \theta$ 
simply by calculating the degree of the rational function $\tilde{\theta}^*x$
of one variable.
%defined in 
%
%\begin{eqnarray*}
%\gamma &:& (x, y)\mapsto (\omega x, -y), \\
%\phi_{E, 2} &:& (x, y)\mapsto  \left(  \frac{2x^2+3x +1}{x-1},  \frac{2 \, \sqrt{2}\, y \,(x^2+3 x +3)}{(x-1)^2}\right).
%\end{eqnarray*}
%
%
\begin{proposition}
Let 
$\gamma: E\to E$ and $\phi_{E, 2}:E\to E$ be the  endomorphisms
defined in Section~\ref{sec:E}.
Then classes of the curves 
\begin{equation*}
\begin{array}{lll}
B_1 =E\times\{P_{\infty}\},
\; &
B_2 =\{P_{\infty}\}\times E,
\; &
B_3 =\Phi_{\id}, 
\; \\
B_4 =\Phi_{\gamma},
\; &
B_5 =\Phi_{\phi_{E, 2}},
\; &
B_6 =\Phi_{ \gamma  \phi_{E, 2}}
\end{array}
\label{eq:sixbasis}
\end{equation*}
 on $A$  form a basis of $S_A$,
 where $P_{\infty}$ is the zero point  of $E$.
\end{proposition}
\begin{proof}
The intersection numbers $\intM{B_i, B_j}{\SA}$ %of these curves 
are given by the following matrix:
\begin{equation}\label{eq:GramSA}
\left[ \begin {array}{cccccc} 0&1&1&1&2&2\\  1&0&1&1&
1&1\\  1&1&0&1&3&4\\  1&1&1&0&2&3
\\  2&1&3&2&0&2\\  2&1&4&3&2&0
\end {array} \right] .
\end{equation}
Since its determinant of  is $-25$, the classes $[B_1], \dots, [B_6]$ form a basis of $\SA$.
%For example, the intersection points of $B_3$ and $B_5$ are 
%$(P_{\infty},P_{\infty})$, $(P\sprime, P\sprime)$ and $(-P\sprime, -P\sprime)$,
%where 
%$$
%P\sprime =(3+2\,\sqrt{2}, 2\,\sqrt{2})\;\in\; E,
%$$
%and the intersection at each of these points is transverse.
%Since the determinant of the matrix~\eqref{eq:GramSA} is $-25$,
%we see that the classes $[B_1], \dots, [B_6]$ form a basis of $\SA$.
\end{proof}
\begin{remark}\label{algebra}
Let $\OOO   = {\End}(E)$ %and $B  = {\End}(E)\otimes {\Q}$ 
be as in Section~\ref{sec:E}.
Set $X  =E\times \{P_{\infty}\} + \{P_{\infty}\}\times E$. 
Then $X$ is a principal polarization on $A$.
For a divisor $L$ on $A$, we have a homomorphism
$$
\begin{array}{cccc}
   \varphi_{L} :  &A  &\to  &{\Pic}^{0}(A) \\
      & x & \mapsto & T_{x}^{*}L - L,
\end{array}
$$
where $T_{x}$ is the translation by $x \in A$ (see~Mumford~\cite{MR0282985}).
We see that $\varphi_{X}^{-1}\circ\varphi_{L}$ is an element of ${\End}(A) = M_{2}(\OOO )$.
We set 
$$
     H = \left\{
\left[
\begin{array}{cc}
a  & b\\
c & d
\end{array}
\right]
~\mid
~a, d \in {\Z},~b, c \in \OOO  ~\mbox{with}~ c = \bar{b}
\right\}.
$$
Then,
$$
\begin{array}{cccc}
j : & S_{A} & \to  & H \\
   & L  &\mapsto  & \varphi_{X}^{-1}\circ\varphi_{L}
\end{array}
$$
is a bijective homomorphism, and  
for $L_{1}, L_{2} \in S_A$ such that
$$
  j(L_{1}) =
\left[
\begin{array}{cc}
a_{1}  & b_{1} \\
c_{1} & d_{1}
\end{array}
\right],~
  j(L_{2}) =
\left[
\begin{array}{cc}
a_{2}  & b_{2} \\
c_{2} & d_{2}
\end{array}
\right],
$$
the intersection number  $\intM{L_{1}, L_{2}}{S_A}$ is given by
$$
\intM{L_{1}, L_{2}}{S_A} = a_{2}d_{1} + a_{1}d_{2} - c_{1}b_{2}
-c_{2}b_{1}.
$$
(see~Katsura~\cite{MR977760}, Katsura~and~Kondo~\cite{MR2862188}).
%Let  $m : E\times E \rightarrow E$ be the addition of $E$, and we set
%$\Delta = \Ker m$.
%We know $ \Delta = \shortset{(P, -P)}{P\in A}$. 
For two endomorphisms
$\alpha_{1}, \alpha_{2}\in \OOO$, by Katsura~\cite{MR977760}\;(also see Katsura~and~Kondo~\cite{MR2862188}), we have
$$
j((\alpha_{1}\times \alpha_{2})^{*} \Delta) =
\left[
\begin{array}{cc}
\bar{\alpha}_{1}\alpha_{1}  & \bar{\alpha}_{1}\alpha_{2}\\
\bar{\alpha}_{2}\alpha_{1} & \bar{\alpha}_{2}\alpha_{2}
\end{array}
\right].
$$
Now consider our basis $[B_{1}], \dots, [B_{6}]$ of $\SA$.
Since we have
%$$
% B_{3}=  (-\id\times \id)^{*}\Delta, B_{4} = (-\gamma\times \id)^{*}\Delta, 
%B_{5}= (-\phi_{E, 2}\times \id)^{*}\Delta, B_{6}=(-\gamma\phi_{E, 2}\times \id)^{*}\Delta
%$$
\begin{eqnarray*}
&&B_{3}=  (-\id\times \id)^{*} \Delta,\;\;  B_{4} = (-\gamma\times \id)^{*} \Delta, \;\;
B_{5}= (-\phi_{E, 2}\times \id)^{*} \Delta,\\
&& B_{6}=(-\gamma\phi_{E, 2}\times \id)^{*} \Delta,
\end{eqnarray*}
we see that
$$
\begin{array}{l}
j(B_{1}) = 
\left[
\begin{array}{cc}
0  & 0\\
0 &  1
\end{array}
\right],
\quad j(B_{2})= 
\left[
\begin{array}{cc}
1  & 0\\
0 &  0
\end{array}
\right], 
\quad j(B_{3})=  
\left[
\begin{array}{cc}
1  & -1\\
-1 &  1
\end{array}
\right],\\
\\
 j(B_{4}) = 
\left[
\begin{array}{cc}
1  & -\gamma^5\\
-\gamma &  1
\end{array}
\right],
\quad j(B_{5})= 
\left[
\begin{array}{cc}
2  & \phi_{E, 2}\\
-\phi_{E, 2} &  1
\end{array}
\right], \\
\\
j(B_{6})=
\left[
\begin{array}{cc}
2  & -\phi_{E, 2}\gamma^2\\
-\gamma\phi_{E, 2} &  1
\end{array}
\right].
\end{array}
$$
Here, as an element in $\OOO$, we use $1$ for $\id$ and $-1$ for $\iota_{E}$.
Using these expressions, we can also calculate our Gram matrix \ref{eq:GramSA} easily.
\end{remark}
From now on,  we express elements of $\SA$ as  row vectors 
with respect to the basis  $[B_1], \dots, [B_6]$.
The matrix~\eqref{eq:GramSA} is then the Gram matrix of $\SA$
with respect to this basis.

%we have the following:
%The intersection number of $\sGamma [\eta]$ and $\Phi_{g}$ is given as follows.
%
%\begin{proposition}\label{intersection}
%Suppose that $\eta$ induces a birational map from 
%$C$ to $\sGamma[\eta]$.
%Then we have 
% $\intM{\sGamma [\eta], \Phi_{g}}{\SA} = \deg \theta$.
% \qed
%\end{proposition}
%
%\begin{proof}
%We denote by  $P_{\infty}$ the zero point of $E$. Then we have
%
%\end{proof}
%
\begin{remark}\label{rem:intersection}
Let $\eta: H \to A$ be as above.
%Suppose that  the morphism $\eta=(\eta_1, \eta_2): H\to A$ above induces a birational map from $H$ to $\sGamma[\eta]$.
Note that
we have 
\begin{equation}\label{eq:degeta}
\intM{\sGamma[\eta], B_1}{\SA}=\deg \eta_{2},
\quad
\intM{\sGamma[\eta], B_2}{\SA}=\deg \eta_{1}.
\end{equation}
By the method above,
we can calculate the vector representation of the class of $\sGamma[\eta]$ in $\SA$
with respect to the basis $[B_1], \dots, [B_6]$.
By the Gram matrix~\eqref{eq:GramSA},
we obtain the self-intersection number of $\sGamma[\eta]$ on $A$.
%If $\eta$ induces a birational morphism from $H$ to
%$\sGamma[\eta]$,
Then  $\sGamma[\eta]$ is smooth
(that is, $\eta$ induces an isomorphism from $H$ to $\sGamma[\eta]$)
if and only if 
\begin{equation}\label{eq:smooth}
\intM{\sGamma[\eta], \sGamma[\eta]}{\SA}
=2(\textrm{the genus of $H$}-1).
\end{equation}
%
%\par
In this case,  we also have
$$
\eta\inv (A_2)=\textrm{the set of fixed points of $\iota_H$},
$$
and hence we can easily obtain the set $\sGamma[\eta]\cap A_2$.
Thus we can calculate the class of the strict transform $\sGamma[\eta]_{\tilA}$ of $\sGamma[\eta]$ in $\StilA$.
\end{remark}
%
%
%
%We demonstrate the method of Remark~\ref{rem:intersection} by some examples.
%
\erase{
\begin{example}
For the Frobenius endomorphism $\Fr\in \End(E)$, we have
$$
\begin{array}{lll}
\intM{\Phi_{\Fr}, B_1}{\SA}=5, &
\intM{\Phi_{\Fr}, B_2}{\SA}=1, &
\intM{\Phi_{\Fr}, B_3}{\SA}=6, \\
\intM{\Phi_{\Fr}, B_4}{\SA}=6, &
\intM{\Phi_{\Fr}, B_5}{\SA}=7, &
\intM{\Phi_{\Fr}, B_6}{\SA}=12.
\end{array}
$$
Since $\Phi_{\Fr} = (-\Fr \times \id)^{*}\Delta$, 
these intersection numbers can be calculated by Remarks~\ref{Frobenius} and~\ref{algebra}.
Hence we obtain
$[\Phi_{\Fr}]=[3, 8, -1, 0, 1, -2]$.
\end{example}
}
\begin{example}\label{example:aut36}
Note that $\Aut(E)$ is a cyclic group of order $6$
generated by $\gamma$. For integers $a$ and $b$,
the pull-back $(\gamma^a\times \gamma^b)^*\Phi_g$ of 
the graph $\Phi_g$ of $g\in \End(E)$ by
the action
$$
(\gamma^a\times \gamma^b): (P, Q)\mapsto (\gamma^a(P),  \gamma^b(Q))
$$
is equal to $\Phi_{\gamma^{-b}g\gamma^a}$.
Calculating the intersection numbers
$\intM{(\gamma^a \times \gamma^b)^* B_i, B_j}{\SA}$, we see that the action 
$(\gamma^a\times \gamma^b)^*$ on $\SA$ is given by
$$
[x_1, \dots, x_6]\mapsto [x_1, \dots, x_6]\cdot  G_1^a \cdot G_2^b,
$$
where
$$
G_1 =
\left[ \begin {array}{cccccc} 1&0&0&0&0&0\\  0&1&0&0&0
&0\\  0&0&0&1&0&0\\  1&1&-1&1&0&0
\\  2&3&-1&0&1&-1\\  1&1&0&-1&1&0
\end {array} \right],
\quad
G_2 =
\left[ \begin {array}{cccccc} 1&0&0&0&0&0\\  0&1&0&0&0
&0\\  1&1&1&-1&0&0\\  0&0&1&0&0&0
\\  1&2&0&0&1&-1\\  0&0&0&0&1&0
\end {array} \right].
$$
\end{example}
\begin{example}\label{example:flip}
In the same way, we see that 
the action of the involution $(P, Q)\mapsto (Q, P)$ of $A$ on $\SA$ is given by
$$
[x_1, \dots, x_6]\mapsto [x_1, \dots, x_6]
\left[ \begin {array}{cccccc} 0&1&0&0&0&0\\  1&0&0&0&0
&0\\  0&0&1&0&0&0\\  1&1&1&-1&0&0
\\  3&3&0&0&-1&0\\  4&4&-1&0&0&-1
\end {array} \right].
$$
\end{example}
\begin{remark}\label{rem:etabar}
Let $\eta: H \to A$ be as above,
and suppose that $\eta$ is an embedding
(that is,
the equality~\eqref{eq:smooth} holds).
Then the induced morphism
$$
\bar{\eta}: H/\gen{\iota_H}=\P^1 \to Y
$$
is an isomorphism from the $u$-line $H/\gen{\iota_H}$ to the $(-2)$-curve  
$\sGamma[\eta]_Y$ on $Y$.
The morphism $\bar{\eta}$ is calculated as follows.
Let $(x_1, y_1)$ and $(x_2, y_2)$ be the affine coordinates of the first and the second factor of $A=E\times E$.
Then the singular surface
$A/\gen{\iota_A}$ is defined by
$$
w^2=(x_1^3-1) (x_2^3-1),
$$
where the quotient morphism $\varpi: A\to A/\gen{\iota_A}$ is given by
$$
((x_1, y_1), (x_2, y_2))\mapsto (x_1, x_2, w)=(x_1, x_2, y_1 y_2).
$$
Then $\rho\circ \bar{\eta}: \P^1\to A/\gen{\iota_A}$
is given by the rational functions 
$$
(\rho\circ \bar{\eta})^*x_1=M_1(u),
\quad
(\rho\circ \bar{\eta})^*x_2=M_2(u),
\quad
(\rho\circ \bar{\eta})^*w=f_H(u) N_1(u)N_2(u).
$$
Let $P$ be a point of $A_2$.
Suppose that the image of $\rho\circ \bar{\eta}$ passes through the node $\varpi(P)$ of $A/\gen{\iota_A}$.
Let $Q\in H$ be the point that is mapped to $P$ by $\eta$, 
and let $Q\sprime\in H/\gen{\iota_H}$ be the image of $Q$ by 
the quotient map $H\to H/\gen{\iota_H}$.
%Note that, 
%since $\eta\circ \iota_H=\iota_A\circ \eta$,
%the point $Q$ is a ramification point of the quotient map $H\to H/\gen{\iota_H}$.
The lift $\bar{\eta}: \P^1\to  Y$ of $\rho\circ\bar\eta$ at $Q\sprime$ is calculated as follows.
Let $T_{P, A}$ denote the tangent space to $A$ at $P$.
Then the $(-2)$-curve $\pi(E_P)=\rho\inv (\varpi(P))$ on $Y$ is canonically identified 
with the projective line $\P_*(T_{P, A})$ of $1$-dimensional linear subspaces of $T_{P, A}$,
and $\bar{\eta}(Q\sprime)\in \pi(E_P)$ corresponds to the image of 
$$
d_Q \eta : T_{Q, H} \to T_{P, A},
$$
where $T_{Q, H}$ is the tangent space to $H$ at $Q$.
Thus $\bar{\eta}(Q\sprime)$ is obtained by differentiating $\eta$ at $Q$.
%\par
In particular,
if $\eta: H\to A$ is defined over $\F_{25}$,
then we can calculate the list of $\F_{25}$-rational points on the $(-2)$-curve $\sGamma[\eta]_Y$ on $Y$.
\end{remark}
We consider the  hyperelliptic curves
defined by
$$
F : v^2=u^6-1,\quand
\quad
G : v^2=\sqrt{2}\, (u^{12}+2\,u^8+2\,u^4+1),
$$
and the morphisms
$$
\renewcommand{\arraystretch}{1.5}
\begin{array}{lll}
\phi_{E, 2}: E\to E && (u, v)\mapsto \left(\dfrac{2u^2+3 u +1}{u-1}, \;\; \dfrac{2\,\sqrt{2}\,  v\,(u^2+3u+3)}{(u-1)^2}\right), \\
\phi_{F , 2}: F \to E && (u,v)\mapsto \left( u^2, \;\; v\right), \\
\phi_{F , 3}: F \to E && (u,v)\mapsto \left(\dfrac{2\, u}{u^3-1},  \;\; \dfrac{ v\,(2\,u^3+1)}{(u^3-1)^2}\right), \\
\phi_{G , 3}: G \to E &&(u, v) \mapsto \left(\dfrac{4\,\sqrt{2}\,(u+3\sqrt{2}+4)^2\, (u+2\,\sqrt{2}+4)}{f}, \;\; \dfrac{(4+4\,\sqrt{2})\,v}{f^2}\right), \\
 && \textrm{where}\quad  f =(u+\sqrt{2})(u+4\,\sqrt{2}+1)\,(u+3\,\sqrt{2}+2), \\
\phi_{G , 4}: G \to E &&(u, v) \mapsto \left(\dfrac{u^4+(1+4\,\sqrt{2})u^2+2}{g}, \;\; \dfrac{v\,u}{g^2}\right), \\
 && \textrm{where}\quad  g =u^4+(1+2\sqrt{2})u^2+(4+\sqrt{2}).
\end{array}
$$
\begin{remark}\label{rem:invH}
Each of 
the five morphisms $\phi: H\to E$ above
satisfies $\iota_E\circ \phi=\phi\circ \iota_H$.
%where $\iota_H$ is the involution of $H$ given by $v\mapsto -v$.
\end{remark}
\begin{remark}
A basis of the vector space ${\rm H}^{0}(G, \Omega^{1}_G)$ of regular 1-forms 
on the curve $G$ is given by
$$
\frac{dx}{y} -\frac{x^4dx}{y},~ \frac{xdx}{y},~\frac{x^3dx}{y},~\frac{x^2dx}{y},~
\frac{dx}{y} + \frac{x^4dx}{y}.
$$
With respect to this basis, the Cartier operator $\CCC$ is given by the matrix
$$
\left[
\begin{array}{c|c}
3 I_3  & O_{3,2} \\
\hline
O_{2,3} & O_{2,2}
\end{array}
\right],
\;
\text{\rm where $I_3$ is the $3\times 3$ identity matrix and $O_{a, b}$ is the $a\times b$ zero matrix.}
\erase{
\left[
\begin{array}{ccccc}
3 &  0 & 0 & 0 & 0  \\
0 &  3 & 0 & 0 & 0  \\
0 &  0 & 3 & 0 & 0  \\
0 &  0 & 0 & 0 & 0  \\
0 &  0 & 0 & 0 & 0
\end{array}
\right].
}
$$
Therefore, we have $\dim \Ker \CCC = 2$ and $\rank \CCC = 3$.
Hence, the Jacobian variety $J(G)$ of $G$ is isogenous to
the product of  a $3$-dimensional ordinary abelian variety and a superspecial
abelian surface $A$. 
In the same way, we see that the Cartier operator
is zero for the curve $F$ and that the Jacobian variety $J(F)$ 
of $F$ is isomorphic to $A$.
\end{remark}
\begin{remark}\label{rem:psis}
The Weierstrass points of  $F$ are $(u, v)=((3+2\,\sqrt{2})^\nu, 0)$ for $\nu=0, \dots, 5$.
%The automorphism group of $F$ is isomorphic to $\Z/2\Z\times S_5$.
%\end{remark}
%
%
%\begin{remark}
The Weierstrass points of  $G$  are $(u, v)=(u_\nu, 0)$ for $\nu=0, \dots, 11$,
where $u_\nu$ are
$$
\pm \sqrt {2}, \;\; \pm 2\,\sqrt {2}, \;\; 1\pm \sqrt {2}, \;\; 2 \pm 2\,\sqrt {2}, \;\; 3 \pm 3\,\sqrt {2},\;\; 4 \pm 4\,\sqrt {2}.
$$
%
%The automorphism group of $G$ is isomorphic to $\Z/2\Z\times S_4$.
%\par
In particular,
let $E\sprime\to \P^1$
(resp.~$\bar{E}\sprime\to \P^1$, $F\sprime\to \P^1$, $G\sprime\to \P^1$)
be the double covering  branched at the points in $P_4$ (resp.~$\bar{P}_4$, $P_6$, $P_{12}$) defined above Theorem~\ref{cor:six}.
Then $E\sprime$ and $\bar{E}\sprime$ are isomorphic to $E$ over $\F_{25}$,
$F\sprime$ is isomorphic to  $F$ over $\F_{25}$, and $G\sprime$ is isomorphic to $G$  over $\F_{25}$.
\end{remark}
We also consider the automorphisms
$$
\renewcommand{\arraystretch}{1.4}
\begin{array}{lll}
\gamma: E\to E  && (u, v)\mapsto (\omega u, -v), \\
h_{F}: F \to F  && (u, v)\mapsto \left(\dfrac{2\sqrt{2}\, u +4 }{u+2\sqrt{2}},\;\; \dfrac{v}{(u+2\sqrt{2})^3}\right), \\
h_{F}\sprime: F \to F  && (u, v)\mapsto \left(\dfrac{2\sqrt{2}\, u +1}{u+3\sqrt{2}},\;\; \dfrac{v}{(u+3\sqrt{2})^3}\right), \\
h_{G}: G \to G  && (u, v)\mapsto \left(\dfrac{2u+3}{u+1},\;\; \dfrac{4v}{(u+1)^6}\right).
\end{array}
$$
%
%\begin{remark}
Note that 
the morphisms $\phi_{E, 2}$ and $\gamma$  have already appeared in Section~\ref{sec:E}.
%\end{remark}
%
\par
Let $\tau$  denote 
the automorphism $(P, Q)\mapsto (Q, \iota_E(P))$ of $A$.
Note that $\tau$ lifts to an automorphism of $\tilA$
and its action on $\StilA$
is obtained from  Examples~\ref{example:aut36},~\ref{example:flip}.
%and the permutation of $A_2$ induced by $\tau$.
For a curve $\varGamma$ on $A$,
we denote by $\TTT(\varGamma)$
the set of translations of $\varGamma$ by points in $A_2$.
Then we define sets of curves on $A$ by 
$$
\renewcommand{\arraystretch}{1.4}
\begin{array}{lll}
\LLL_{01}& =&\TTT(\sGamma[(\phi_{F , 2}, \phi_{F , 2} h_{F})]\,),\\
\LLL_{02}& =&\TTT(\sGamma[(\phi_{F , 3}, \phi_{F , 3} h_{F}\sprime]\,),\\
\hspace*{0.5cm}\LLL_{10, (4,3)} & =& \TTT(\sGamma[(\phi_{G , 4}, \phi_{G , 3})]\,), \\ 
\hspace*{0.5cm}\LLL_{10, (4,4)} & =& \TTT(\sGamma[(\gamma^2 \phi_{G , 4},  \gamma \phi_{G , 4} h_G)]\,), \\
\LLL_{10} & =&\LLL_{10, (4,3)} \cup \tau(\LLL_{10, (4,3)}) \cup \LLL_{10, (4,4)} \cup \tau(\LLL_{10, (4,4)}), \\
\hspace*{0.5cm}\LLL_{11, (1,2)} & =&\TTT(\sGamma[(\gamma^2, \gamma^2 \phi_{E, 2})]\,), \\
\hspace*{0.5cm}\LLL_{11, (2,2)} & =&\TTT(\sGamma[ ( \phi_{E, 2} \gamma, \gamma  \phi_{E, 2})]\,), \\
\LLL_{11} & =&\LLL_{11, (1,2)} \cup \tau(\LLL_{11, (1,2)}) \cup \LLL_{11, (2,2)} \cup \tau(\LLL_{11, (2,2)}), \\
\LLL_{12} & =& \TTT(B_1) \cup \TTT(B_2) \cup \TTT(B_4) \cup \TTT(\sGamma[(\id, \gamma^2)]\,). 
\end{array}
$$
Using the method above, %Proposition~\ref{intersection}, 
we have the following list of intersection numbers.
\begin{center}
\begin{tabular}{| l |c|c|c|c|c|c|}
\hline
     & $B_{1}$ & $B_{2}$ & $B_{3}$ & $B_{4}$ & $B_{5}$ & $B_{6}$\\
\hline
$\sGamma[(\phi_{F , 2}, \phi_{F , 2} h_{F})]$ 
 & 2& 2 &4  & 2& 8 & 7 \\
\hline
$\sGamma[(\phi_{F , 3}, \phi_{F , 3} h_{F}\sprime)]$ 
&3 & 3 & 6& 3& 5& 12 \\
\hline
$\sGamma[(\phi_{G , 4}, \phi_{G , 3})]$
 & 3 & 4 & 7&  4&  14& 15 \\
\hline
$\sGamma[(\gamma^2 \phi_{G , 4},  \gamma \phi_{G , 4} h_G)]$ 
&  4& 4 & 7& 3 & 14 & 16\\
\hline
$\sGamma[(\gamma^2,  \gamma^2 \phi_{E , 2})]$
 & 2 & 1 &3 & 2 & 3 & 7 \\
\hline
$\sGamma[( \phi_{E, 2} \gamma, \gamma  \phi_{E, 2})]$
& 2& 2 & 5& 2& 6& 8\\
\hline
$\sGamma[(\id, \gamma^2)]$
 & 1 & 1 & 3 & 1 & 2  & 2 \\
\hline
\end{tabular}
\end{center}
Using this table and the Gram matrix~\eqref{eq:GramSA},
we obtain the following vector representations of classes of these curves.
$$
\renewcommand{\arraystretch}{1.5}
\begin{array}{lll}
{}[\sGamma[(\phi_{F , 2}, \phi_{F , 2} h_{F})]\,] & = & [2,3,-1,2,-1,0] \\
{}[\sGamma[(\phi_{F , 3}, \phi_{F , 3} h_{F}\sprime)]\,] & = & [4,6,-2,3,-1,-1] \\
{}[\sGamma[(\phi_{G , 4}, \phi_{G , 3})]\,] & = & [5,6,-2,3,-1,-1] \\
{}[\sGamma[(\gamma^2 \phi_{G , 4}, \gamma \phi_{G , 4} h_G)]\,] & = &  [4,6,-2,4,-1,-1] \\
{}[\sGamma[(\gamma^2, \gamma^2 \phi_{E, 2})]\,] & = &  [2,4,-1,1,0,-1]\\
{}[\sGamma[ ( \phi_{E, 2} \gamma, \gamma  \phi_{E, 2})]\,]&=& [3,4,-2,2,0,-1]\\
{}[\sGamma[(\id,  \gamma^2)]\,]&=&[1,1,-1,1,0,0].
\end{array}
$$
\begin{remark}\label{rem:smooth}
In particular,
we see  that these curves are smooth
by confirming~\eqref{eq:smooth}. 
\end{remark}
\begin{remark}
Incidentally, by the vector representations of classes of our curves, we have
$$
\begin{array}{l}
j(\Gamma[(\phi_{F,2}, \phi_{F, 2} h_{F})])=
\left[
\begin{array}{cc}
2  & 1 + 2\gamma^{2} - \phi_{E, 2}\\
1 - 2\gamma +  \phi_{E, 2} &  2
\end{array}
\right], \\
\\
j(\Gamma[(\phi_{F,3}, \phi_{F, 3} h_{F}')])=
\left[
\begin{array}{cc}
3  & 2 + 3\gamma^{2} - \phi_{E, 2} + \phi_{E, 2}\gamma^{2}\\
1 - 3\gamma +  \phi_{E, 2} + \gamma\phi_{E, 2} &  3
\end{array}
\right], \\
\\
j(\Gamma[(\phi_{G,4}, \phi_{G, 3})])=
\left[
\begin{array}{cc}
3  & 1 + 3\gamma^{2} - \phi_{E, 2} + \phi_{E, 2}\gamma^{2} \\
1 - 3\gamma +  \phi_{E, 2} +  \gamma\phi_{E, 2} &  4
\end{array}
\right], \\
\\
j(\Gamma[(\gamma^{2}\phi_{G,4}, \gamma\phi_{G, 4} h_{G})])=
\left[
\begin{array}{cc}
4  & 2 + 4\gamma^{2} - \phi_{E, 2} + \phi_{E, 2}\gamma^{2} \\
2 - 4\gamma +  \phi_{E, 2} +  \gamma\phi_{E, 2} &  4
\end{array}
\right], \\
\\
j(\Gamma[(\gamma^{2}, \gamma^{2}\phi_{E, 2})])=
\left[
\begin{array}{cc}
2  & 1 + \gamma^{2} + \phi_{E, 2}\gamma^{2}\\
1 - \gamma +  \gamma\phi_{E, 2} &  1
\end{array}
\right], \\
\\
j(\Gamma[(\phi_{E,2}\gamma, \gamma\phi_{E, 2})])=
\left[
\begin{array}{cc}
2  & 2 + 2\gamma^{2} - \phi_{E, 2}\gamma^{2}\\
2 - 2\gamma +  \gamma\phi_{E, 2} &  2
\end{array}
\right], \\
\\
j(\Gamma[({\id}, \gamma^{2})])=
\left[
\begin{array}{cc}
1  & \gamma\\
-\gamma^{2} &  1
\end{array}
\right]. 
\end{array}
$$
We can also use these expressions to calculate the intersection numbers.
\end{remark}
Now we state our main result of this section.
\begin{theorem}\label{thm:sec9main}
For  $\nu i=01,02,10,11,12$, 
the set  
$$
\SSS_{\nu i} =\set{\Gamma_Y}{\Gamma\in \LLL_{\nu i}}
$$
is a set of disjoint $16$  smooth rational curves on $Y$.
Moreover, together with the set 
$\SSS_{00}$ of the  images of the $(-1)$-curves  $E_P$  for $P\in A_2$ by $\pi: \tilA\to Y$,
the six sets  
$\SSS_{00}, \SSS_{01}, \SSS_{02}, \SSS_{10}, \SSS_{11}, \SSS_{12}$
 satisfy the conditions {\rm (a)}, {\rm (b)} and  {\rm (c)} in Theorem~\ref{thm:six} and 
 possess the properties in Theorem~\ref{cor:six}.
\end{theorem}
\begin{proof}
Let $\SSS$ be the union of the six sets $\SSS_{00}, \SSS_{01}, \SSS_{02}, \SSS_{10}, \SSS_{11}, \SSS_{12}$.
We have already seen that the $96$ curves in $\SSS$ are $(-2)$-curves on $Y$ 
(see Remarks~\ref{rem:invH} and~\ref{rem:smooth}).
Since the $96$ rational curves in $\SSS$ are presented explicitly,
we can prove Theorem~\ref{thm:sec9main} by direct computation.
%We present a part of the computation by an example.
\par
By the method in Remark~\ref{rem:intersection},
we can calculate the classes
$[\Gamma_Y] \in S_Y$ of the $96$ rational curves $\Gamma_Y\in \SSS$:
more precisely,
we calculate the vector representations of 
the classes 
$[\pi^* (\Gamma_Y)]$
of the curves $\pi^* (\Gamma_Y)$ on $\tilA$
with respect to the basis $[B_1], \dots, [B_6]$ and $[E_P]$ ($P\in A_2$).
Using the Gram matrix~\eqref{eq:GramSA} and the formula
$$
\intM{[\Gamma_Y], [\Gamma\sprime_Y]}{S_Y}=
\frac{1}{2} \intM{[\pi^*(\Gamma_Y)], [\pi^*(\Gamma\sprime_Y)]}{S_{\tilA}},
$$
we can calculate the intersection numbers among the curves in $\SSS$.
It follows that the six sets $\SSS_{\nu i}$ 
satisfy the conditions {\rm (a)}, {\rm (b)} and  {\rm (c)} in Theorem~\ref{thm:six}.
\par
Next we calculate the list $\Gamma_Y(\F_{25})$
of $\F_{25}$-rational points by the method in Remark~\ref{rem:etabar}. 
It turns out that
$$
\intM{[\Gamma_Y], [\Gamma\sprime_Y]}{S_Y}=|\Gamma_Y(\F_{25}) \cap \Gamma\sprime_Y(\F_{25})|
$$
for any pair $\Gamma_Y, \Gamma\sprime_Y$ of distinct curves in $\SSS$.
Therefore any intersection point of curves in $\SSS$ is an $\F_{25}$-rational point.
Moreover the properties in Theorem~\ref{cor:six} can be checked directly.
\par
For example,
we consider a curve $\sGamma[\eta] \in \LLL_{10, (4,4)}$,
where 
the morphism $\eta: G\to A$ is given by 
\begin{eqnarray*}
\eta^*x_1 &=& {\frac { \left( 2+2\,\sqrt {2} \right) \left( {u}^{2}+ \left( 4+3\,\sqrt {2} \right) u+4\,\sqrt {2
} \right)  \left( {u}^{2}+ \left( 1+2\,\sqrt {2} \right) u+4\,\sqrt {2
} \right)   }{ \left( u+4\,\sqrt {2}
 \right)  \left( u+3\,\sqrt {2}+3 \right)  \left( u+2\,\sqrt {2}+2
 \right)  \left( u+\sqrt {2} \right) }}, \\
 \eta^*y_1 &=& {\frac {uv}{ \left( u+4\,
\sqrt {2} \right) ^{2} \left( u+3\,\sqrt {2}+3 \right) ^{2} \left( u+2
\,\sqrt {2}+2 \right) ^{2} \left( u+\sqrt {2} \right) ^{2}}}, \\
 \eta^*x_2 &=&  {\frac 
{ \left( 4+3\,\sqrt {2} \right) \left( {u}^{2}+ \left( 3+4\,\sqrt {2} \right) u+3\,\sqrt {2}+4
 \right)  \left( {u}^{2}+ \left( 1+4\,\sqrt {2} \right) u+4\,\sqrt {2}
+3 \right)   }{ \left( u+4+\sqrt {2}
 \right)  \left( u+\sqrt {2}+1 \right)  \left( u+2\,\sqrt {2}+2
 \right)  \left( u+3\,\sqrt {2}+2 \right) }},\\
 \eta^*y_2 &=&  {\frac { \left( 1+\sqrt {2} \right) v \left( u+4
 \right)  \left( u+1 \right)  }{ \left( u+
4+\sqrt {2} \right) ^{2} \left( u+\sqrt {2}+1 \right) ^{2} \left( u+2
\,\sqrt {2}+2 \right) ^{2} \left( u+3\,\sqrt {2}+2 \right) ^{2}}}.
\end{eqnarray*}
%
%which is a member of $\LLL_{10, (4,4)}$.
The vector representation of $[\sGamma[\eta]_{\tilA}]\in S_{\tilA}$ is 
$$
[\sGamma[\eta]_{\tilA}]=[4,6,-2,4,-1,-1]-\sum_{P\in T[\eta]}[E_P],
$$
where $[4,6,-2,4,-1,-1] \in S_A$ is written with respect to $[B_1], \dots, [B_6]$,
and
$$
T[\eta]=\{P_{{\infty \infty }},P_{{\infty 0}},
P_{{\infty 1}},P_{{\infty 2}},
P_{{0\infty }},P_{{00}},P_{{01}},
P_{{02}},P_{{1\infty }},
P_{{12}},P_{{2\infty }},P_{{22}}\}.
$$
Here $P_{\alpha\beta}$ denotes $(P_\alpha, P_\beta)\in A_2$ for $\alpha, \beta\in \{\infty, 0, 1, 2\}$
(see~Section~\ref{sec:E}).
The induced isomorphism $\bar{\eta}$ from the $u$-line $\P^1=G/\gen{\iota_G}$ to 
the $(-2)$-curve $\sGamma[\eta]_Y\in \SSS_{10}$ induces the  bijection between the sets
of $\F_{25}$-rational points given in Table~\ref{table:examplebareta}.
In this table, 
the point $\bar{\eta}(u)$ is written by the following method:
If $\bar{\eta}(u)$ is not 
on the exceptional divisor of $\rho$,
then the 
coordinates $[x_1, x_2, w]$ of $\bar{\eta}(u)$ on 
$A/\gen{\iota_A}$ defined by  $w^2=(x_1^3-1)(x_2^3-1)$
 is given. (See Remark~\ref{rem:etabar}.)
If $\bar{\eta}(u)$ is 
on the $(-2)$-curve  $\pi(E_P)=\rho\inv(\varpi(P))$
corresponding to $P\in A_2$,
then the point $\bar{\eta}(u)$  is written by the 
coordinates $[[x_1, x_2], [\xi_0, \xi_1]]$,
where $[\xi_0, \xi_1]$ is the homogeneous coordinates
on $\pi(E_P)=\rho\inv (\varpi(P))\cong \P_*(T_{P, A})$ with respect 
to the basis $\tilde\theta_P, \tilde\theta_P$ of $T_{P, A}$,
where $ \tilde\theta$ is a non-zero  invariant vector field on $E$,
which is unique up to scalar multiplications.
\begin{table}
{\small
$$
\begin{array}{lcl}
\bar{\eta}(\infty) &=&[2+2\,\sqrt{2}, 4+3\,\sqrt{2}, 0], \\
\bar{\eta}(0) &=&[2+3\,\sqrt{2}, 4+3\,\sqrt{2}, 0], \\
\bar{\eta}(1) &=&[1+3\,\sqrt{2}, 1, 0], \\
\bar{\eta}(2) &=&[1+3\,\sqrt{2}, 4+3\,\sqrt{2}, 4+3\,\sqrt{2}], \\
\bar{\eta}(3) &=&[1+3\,\sqrt{2}, 4+3\,\sqrt{2}, 1+2\,\sqrt{2}], \\
\bar{\eta}(4) &=&[1+3\,\sqrt{2}, 2+3\,\sqrt{2}, 0], \\
\bar{\eta}(\sqrt{2}) &=&[[\infty, 1], [1, 2\,\sqrt{2}]], \\
\bar{\eta}(1+\sqrt{2}) &=&[[2+3\,\sqrt{2}, 2+2\,\sqrt{2}], [1, 2]], \\
\bar{\eta}(2+\sqrt{2}) &=&[2\,\sqrt{2}, 2+\sqrt{2}, 3], \\
\bar{\eta}(3+\sqrt{2}) &=&[4+4\,\sqrt{2}, 4+\sqrt{2}, 3], \\
\bar{\eta}(4+\sqrt{2}) &=&[[2+2\,\sqrt{2}, 2+2\,\sqrt{2}], [1, 4+\sqrt{2}]], \\
\bar{\eta}(2\,\sqrt{2}) &=&[[1, 2+3\,\sqrt{2}], [1, 4+2\,\sqrt{2}]], \\
\bar{\eta}(1+2\,\sqrt{2}) &=&[3\,\sqrt{2}, 2+\sqrt{2}, 1+\sqrt{2}], \\
\bar{\eta}(2+2\,\sqrt{2}) &=&[[\infty, 2+2\,\sqrt{2}], [1, 2\,\sqrt{2}]], \\
\bar{\eta}(3+2\,\sqrt{2}) &=&[[1, \infty], [1, 2+\sqrt{2}]], \\
\bar{\eta}(4+2\,\sqrt{2}) &=&[3+4\,\sqrt{2}, 4+\sqrt{2}, 4+4\,\sqrt{2}], \\
\bar{\eta}(3\,\sqrt{2}) &=&[[1, 1], [1, \sqrt{2}]], \\
\bar{\eta}(1+3\,\sqrt{2}) &=&[3+4\,\sqrt{2}, 2+4\,\sqrt{2}, 1+4\,\sqrt{2}], \\
\bar{\eta}(2+3\,\sqrt{2}) &=&[[1, 2+2\,\sqrt{2}], [1, \sqrt{2}]], \\
\bar{\eta}(3+3\,\sqrt{2}) &=&[[\infty, \infty], [1, 4+2\,\sqrt{2}]], \\
\bar{\eta}(4+3\,\sqrt{2}) &=&[3\,\sqrt{2}, 3+4\,\sqrt{2}, 3], \\
\bar{\eta}(4\,\sqrt{2}) &=&[[\infty, 2+3\,\sqrt{2}], [1, 3+4\,\sqrt{2}]], \\
\bar{\eta}(1+4\,\sqrt{2}) &=&[[2+2\,\sqrt{2}, \infty], [1, 1]], \\
\bar{\eta}(2+4\,\sqrt{2}) &=&[4+4\,\sqrt{2}, 2+4\,\sqrt{2}, 4+4\,\sqrt{2}], \\
\bar{\eta}(3+4\,\sqrt{2}) &=&[2\,\sqrt{2}, 3+4\,\sqrt{2}, 1+4\,\sqrt{2}], \\
\bar{\eta}(4+4\,\sqrt{2}) &=&[[2+3\,\sqrt{2}, \infty], [1, 2+2\,\sqrt{2}]]
\end{array}
$$
}
\caption{The map $\bar{\eta}$ on $\F_{25}$-rational points}\label{table:examplebareta}
\end{table}%
\par
We put $\varGamma =\sGamma[\eta]_Y$,
and present the four subsets $\varGamma_{1}, \varGamma_{00}, \varGamma_{01}, \varGamma_{02}$
of $\varGamma(\F_{25})$ in Theorem~\ref{cor:six}.
The set  $\varGamma_{00}$ of $12$ points on the exceptional divisor of $\rho$ is easily obtained from Table~\ref{table:examplebareta}.
The other sets 
are given as follows:
\begin{eqnarray*}
\bar\eta\inv (\varGamma_{1}) &=& \{\infty ,0,1,2,3,4\}, \\
\bar\eta\inv (\varGamma_{01}) &=& \{3+\sqrt {2},4+2\,\sqrt {2},1+3\,\sqrt {2},2+4\,\sqrt {2}\}, \\
\bar\eta\inv (\varGamma_{02}) &=& \{2+\sqrt {2},1+2\,\sqrt {2},4+3\,\sqrt {2},3+4\,\sqrt {2}\}.
\end{eqnarray*}
For example, the  unique  $(-2)$-curve in $\SSS_{11}$ passing through $\bar\eta(\infty) \in \varGamma_1$ is $\sGamma[\eta\sprime]_Y$,
where 
$\eta\sprime: E\to A$ is  given by 
$$
\left[
\left[{\frac {{u}^{2}+ \left( 1+3\,\sqrt {2} \right) u+2\,\sqrt {2}+1}{
 \left( u+3\,\sqrt {2}+4 \right) ^{2}}},{\frac { \left( 4+2\,\sqrt {2} \right) v  \left( u+2\,\sqrt {2}
+4 \right) }{ \left( u+3\,\sqrt {2}+4
 \right) ^{3}}}
 \right],
 \left[\left( 2+2\,\sqrt {2} \right) u  ,4\,v  \right] \right],
$$
and we have $\bar\eta\sprime(1+3\sqrt{3})=\bar\eta(\infty)$,
while 
the unique $(-2)$-curve in $\SSS_{12}$ passing through $\bar\eta(\infty) \in \varGamma_1$  
is $\sGamma[\eta\spprime]_Y$,
where 
$\eta\spprime: E\to A$ is  given by 
$$
[[2+2\,\sqrt {2},0],[u,v]], 
$$
and we have $\bar\eta\spprime(4+3\,\sqrt{2})=\bar\eta(\infty)$.
The  unique  $(-2)$-curve in $\SSS_{01}$ passing through $\bar\eta(3+\sqrt {2}) \in \varGamma_{01}$ is $\sGamma[\xi]_Y$,
where 
$\xi: F\to A$ is  given by
$$
\left[\left[{u}^{2},v\right], 
\left[{\frac {3\, \left( u+\sqrt {2} \right) ^{2}}{ \left( u+2
\,\sqrt {2} \right)^{2}}},{\frac {v}{ \left( u+2\,\sqrt {2} \right) ^
{3}}}\right]\right]
$$ 
and we have $\bar\xi(4+3\,\sqrt{2})=\bar\eta(3+\sqrt {2})$.
\erase{
The  unique  $(-2)$-curve in $\SSS_{02}$ passing through $\bar\eta(2+\sqrt {2}) \in \varGamma_{02}$ is $\sGamma[\xi\sprime]_Y$,
where 
$\xi\sprime: F\to A$ is  given by
\begin{eqnarray*}
&&\left[\left[
{\frac { \left( 1+4\,\sqrt {2} \right) \left( {u}^{3}+u \left( 2+2\,\sqrt {2} \right) +4 \right) 
 }{ \left( u+4+\sqrt {2} \right) ^{2}
 \left( u+3\,\sqrt {2}+2 \right) }},
{\frac { 4\,v\, \left( u+4\,\sqrt {2}+4 \right)  \left( u+2 \right) }{ \left( u+3\,\sqrt {2}+2 \right) ^{2}
 \left( u+4+\sqrt {2} \right) ^{3}}}\right], \right. \\ &&
\left.\left[{\frac {\left( 3+3\,\sqrt {2} \right) \left( {u}^{3}+
 \left( 2+3\,\sqrt {2} \right) {u}^{2}+2\,u+1+\sqrt {2} \right) 
  }{ \left( u+3\,\sqrt {2}+2 \right) 
 \left( u+4\,\sqrt {2}+2 \right) ^{2}}},{\frac { 2\,\sqrt {2}\,v\,\left( u+3
 \right)  \left( u+2\,\sqrt {2}+4 \right) }{ \left( u+4\,
\sqrt {2}+2 \right) ^{3} \left( u+3\,\sqrt {2}+2 \right) ^{2}}}
\right]\right]
\end{eqnarray*}
and we have $\bar\xi\sprime(4+2\,\sqrt {2})=\bar\eta(3+\sqrt {2})$.
}
\par
The details of  these data for all $96$ curves in $\SSS$ are presented in~\cite{ShimadaCompData}.
\end{proof}
We give a remark about  $(16_r)$-configurations on a $K3$ surface in general.
\begin{proposition}\label{prop,g}
Assume that the characteristic $p$ of the base field is $\neq 2$.
No abelian surfaces contain any non-singular hyperelliptic curve of genus 
greater than or equal to $6$.
\end{proposition}
\begin{proof}
Suppose that an abelian surface $A$ contains a nonsingular hyperelliptic curve $C$
of genus $g$. We may assume that $C$ is symmetric under the inversion $\iota$
of $A$. 
Then, $C \cap A_{2}$ must contain $2g + 2$ points. 
Since the number of
points in $A_{2}$ is $16$, we have $g \leq 7$. Assume $g = 7$. Then, we have
$C \cap A_{2} = A_{2}$. 
If there exists a two-torsion point $x$ such that
$T_{x}^{*}C\neq C$, then we have $C^2 = (C, T_{x}^{*}C)\geq 16$.
Therefore, the genus of $C$ is greater than or equal to $16/2 + 1 = 9$, which
contradicts $g = 7$. Suppose $T_{x}^{*}C = C$ for any $x \in A_{2}$.
Then, the group scheme $K(C) = \Ker \varphi_{C}$ contains $A_{2}$,
where $\varphi_C$ is defined in Remark~\ref{algebra}.
On the other hand, by the Riemann-Roch theorem, we have 
$$
| K(C) | =\deg \varphi_{C} = (C^2/2)^2 = (g - 1)^2 = 36.
$$
Since $A_{2} \subset  K(C)$, $36$ must be divisible by $16$, a contradiction.
Hence, $A$ does not contain any nonsingular hyperelliptic curve of genus $7$.

Now, assume $g = 6$. Then, since $C$ is hyperelliptic,
we have $| C \cap A_{2}| = 2\times 6 + 2 = 14$.
Let $x$ be a point in $A_{2}$ that is not contained in $C \cap A_{2}$.
Take a point $y \in C \cap A_{2}$. Then, we have $C \neq T_{x-y}^{*}C$
and $C \cap T_{x-y}^{*}C \cap A_{2}$ contains more than or equal to $12$ points.
Therefore, we have $C^2 = (C, T_{x-y}^{*}C)\geq 12$. 
Hence, the genus
of $C$ must be greater than or equal to $12/2 + 1 = 7$, which contradicts $g = 6$.
Consequently $A$ does not contain any nonsingular hyperelliptic curve of genus $6$.
\end{proof}

\begin{remark}
Let $C$ be a nonsingular complete curve of genus $2$, and let $J(C)$ be a Jacobian
variety. 
Then, it is well-known that on the Kummer surface $\Km(J(C))$ there exists
a $(16_{6})$-configuration. We also have  a $(16_{10})$-configuration on some Kummer
surfaces, using a certain hyperelliptic curve of genus $4$ (see~Traynard~\cite{MR1509078},
Barth~and~Nieto~\cite{MR1257320}, Katsura~and~Kondo~\cite{MR2862188}). 
In this paper, we constructed 
a $(16_{12})$-configuration on the supersingular  $K3$ surface with Artin invariant $1$
in characteristic 5.
This seems to be the first example of $(16_{12})$-configurations on a $K3$ surface. 
To construct the configuration we use a hyperelliptic curve of genus $5$. 
By Proposition~\ref{prop,g}, we cannot construct $(16_{2\ell})$-configurations with $\ell \geq 7$
on a Kummer surface in a similar way to our method.
\end{remark}	
\begin{remark}
The supersingular K3 surface with Artin invariant $1$ in characteristic $5$ has 
%an interesting pencil of curves of genus $2$. 
an interesting example of a pencil of curves of genus $2$. 
Let $P$ be a point of $\P^2(\F_{25})\setminus C_F(\F_{25})$, 
and let $R_1$ and $R_2$ be the points on $X$ that are mapped to $P$ by $\pi_F: X\to \P^2$.
We take the blowing-up $\tilde{X}$ at the two points $R_{1}$, $R_{2}$ of $X$. 
Then, the pencil of lines passing through $P$ induces on 
$\tilde{X}$ a structure of fiber space 
over ${\P}^{1}$ whose general fiber is isomorphic to a smooth complete curve $C$
of genus $2$ defined by $y^2 = x^6 -1$.
The fiber space has exactly $6$ degenerate fibers corresponding to the tangent lines 
of $C_F$ passing through $P$.
Each degenerate fiber is a union of two smooth rational curves intersecting at one point 
with multiplicity $3$.

Let $C_{1}$ be the nonsingular complete model of the curve defined by
the equation $1 + x_{1}^6 + x_{2}^6 = 0$.
$G ={\Z}/6{\Z} = \langle \theta \rangle$ with a generator $\theta$.
We denote by $\xi$ a primitive $6$-th root of unity, and consider the action
$$
\begin{array}{clcllcl}
\theta : &x_{1} &  \mapsto & x_{1}, &  x_{2} & \mapsto & \xi x_{2} \\
    &  x & \mapsto & \xi x, &  y & \mapsto & y
\end{array}
$$
on the surface $C_{1}\times C$.  The group $G$ also acts on the curve $C_{1}$.
We set
%$$
%  x_{1}, ~w = \sqrt{-1}(x_{2}/x)^{3}y,~z = x_{2}/x.
%$$
$$
w = \sqrt{-1}(x_{2}/x)^{3}y,~z = x_{2}/x.
$$
Then, $x_{1}, w$ and $z$ are $G$-invariant and the quotient surface $(C_{1}\times C)/G$
is birationally isomorphic to the surface defined by $w^2 = z^{6} + 1 + x_{1}^6$.
The fiber space structure is given by $(C_{1}\times C)/G \rightarrow C_{1}/G$.
\end{remark}

\bibliographystyle{plain}

\def\cftil#1{\ifmmode\setbox7\hbox{$\accent"5E#1$}\else
  \setbox7\hbox{\accent"5E#1}\penalty 10000\relax\fi\raise 1\ht7
  \hbox{\lower1.15ex\hbox to 1\wd7{\hss\accent"7E\hss}}\penalty 10000
  \hskip-1\wd7\penalty 10000\box7} \def\cprime{$'$} \def\cprime{$'$}
  \def\cprime{$'$} \def\cprime{$'$}

\end{document}